\documentclass[11pt,oneside]{amsart}

\usepackage{amsmath,ifthen, amsfonts, amssymb,srcltx,amsopn, color, enumerate, 
hyperref}
\usepackage[cmtip,arrow]{xy}
\usepackage{pb-diagram, pb-xy}
\usepackage{overpic}

\usepackage[T1]{fontenc}

\newcommand{\showcomments}{yes}

\newsavebox{\commentbox}
{\ifthenelse{\equal{\showcomments}{yes}}%
{\footnotemark
        \begin{lrbox}{\commentbox}
        \begin{minipage}[t]{1.25in}\raggedright\sffamily\tiny
        \footnotemark[\arabic{footnote}]}
{\begin{lrbox}{\commentbox}}}%
{\ifthenelse{\equal{\showcomments}{yes}}%
{\end{minipage}\end{lrbox}\marginpar{\usebox{\commentbox}}}
{\end{lrbox}}}

\newcounter{intronum}
\newcounter{ax}
\newtheorem{thm}{Theorem}[section]

\newtheorem{lem}[thm]{Lemma}

\newtheorem{cor}[thm]{Corollary}

\newtheorem{prop}[thm]{Proposition}

\newtheorem{propi}[intronum]{Proposition}

\theoremstyle{definition}
\newtheorem{defn}[thm]{Definition}
\newtheorem{rem}[thm]{Remark}

\newtheorem{question}[thm]{Question}

\newtheorem{claim*}{Claim}

\DeclareMathOperator{\dimension}{dim}

\DeclareMathOperator{\Aut}{Aut}

\DeclareMathOperator{\stabilizer}{Stab}
\DeclareMathOperator{\diam}{diam}

\newcommand{\neb}{\mathcal N}

\newcommand{\field}[1]{\mathbb{#1}}
\newcommand{\integers}{\ensuremath{\field{Z}}}

\newcommand{\naturals}{\ensuremath{\field{N}}}
\newcommand{\reals}{\ensuremath{\field{R}}}
\newcommand{\Euclidean}{\ensuremath{\field{E}}}

\newcommand{\boundary}{{\ensuremath \partial}}

\makeatletter

\newcommand{\Rmnum}[1]{\mathbf{{\expandafter\@slowromancap\romannumeral #1@}}}

\newcommand{\contact}[1]{\ensuremath{\mathcal C#1}}

\makeatother

\let\oldmarginpar\marginpar
\renewcommand\marginpar[1]{\-\oldmarginpar[\raggedleft\footnotesize #1]%
{\raggedright\footnotesize #1}}

\newcounter{enumitemp}

\newcommand{\dist}{\textup{\textsf{d}}}

\newcommand{\OL}{\overleftarrow}
\newcommand{\OR}{\overrightarrow}

\newcommand{\gate}{\mathfrak g}

\begin{document}
\title{Large facing tuples and a strengthened sector lemma}
\author[M Hagen]{Mark Hagen}
\address{School of Mathematics, University of Bristol, Bristol, UK}
\email{markfhagen@posteo.net}

\setcounter{tocdepth}{1}
\date{\today}
\maketitle

\begin{abstract}
We prove a strengthened sector lemma for irreducible, finite-dimensional, locally finite, essential, cocompact 
CAT(0) cube 
complexes under the additional hypothesis that the complex is \emph{hyperplane-essential}; we prove that every 
quarterspace contains a halfspace.  In aid of this, we present simplified proofs of known results about 
loxodromic isometries of the contact graph, avoiding the use of disc diagrams.

This paper has an expository element; in particular, we collect results about cube complexes 
proved by combining Ramsey's theorem and Dilworth's theorem.  We illustrate the use of these tricks with a 
discussion of the Tits alternative for cubical groups, and ask some questions about ``quantifying'' statements 
related to rank-rigidity and the Tits alternative.
\end{abstract}

\tableofcontents

\section{Introduction}\label{sec:intro}
In various guises, CAT(0) cube complexes appear throughout mathematics.  They appear in discrete mathematics 
as \emph{median graphs} (the equivalence of median graphs and CAT(0) cube complexes is due to 
Chepoi~\cite{Chepoi}) and many other equivalent combinatorial structures: discrete median 
algebras~\cite{avann1961median,Roller}, event structures~\cite{nielsen1981petri,barthelemy1993median}, etc. 
(see Bandelt-Chepoi~\cite{bandelt2008metric} for a survey).  After being introduced into group theory by Gromov 
as a source of examples~\cite{Gromov:essay}, CAT(0) cube complexes were understood by Sageev to provide the 
correct generalisation of trees needed to formulate a ``high-dimensional Bass-Serre 
theory''~\cite{Sageev95,sageev1997codimension}.  

The theory of groups acting on CAT(0) cube complexes has since proved 
extremely useful.  Nonpositively-curved cube complexes provide the setting for Wise's cubical 
small-cancellation theory~\cite{Wise:pave}, and the sub-class of \emph{special} cube complexes defined by 
Haglund-Wise~\cite{haglund2008special} provides a class of groups with many separable subgroups.  These ideas 
were crucial to the resolution of several conjectures about $3$--manifolds, notably the virtual Haken and 
virtual fibering conjectures~\cite{agol2013virtual}.

CAT(0) cube complexes are extremely organised spaces in which one has many tools far beyond CAT(0) geometry, 
largely because of the median structure, the \emph{hyperplanes}, and the \emph{(combinatorially) convex 
subcomplexes}.  This has strong coarse-geometric consequences, e.g. finite asymptotic 
dimension~\cite{wright2012finite} and various results related to quasi-isometric rigidity, 
e.g.~\cite{huang2017top,huang2018groups}.  The nice geometric features of CAT(0) cube complexes have also led 
to the study of non-cubical spaces that can be ``approximated'' by cube complexes in one way or another, as in 
coarse median spaces~\cite{bowditch2013coarse,bowditch2018convex,bowditch2019quasiflats} and hierarchically 
hyperbolic spaces~\cite{behrstock2017quasiflats}.  

The purpose of this paper is threefold.  First, we prove a statement, Proposition~\ref{propi:SSL}, 
about hyperplane-essential actions on CAT(0) cube complexes, needed elsewhere in the literature.  
In~\cite[Lemma 5.2]{CapraceSageev}, Caprace and Sageev show that, 
given an $\Aut(X)$--\emph{essential}, irreducible cube complex $X$ on which $\Aut(X)$ acts without a global 
fixed point 
at infinity, and given crossing hyperplanes $v,h$, there exist disjoint hyperplanes $a,b$ such that 
$a$ and $b$ are separated by both $v$ and $h$.  This is crucial for their proof of rank-rigidity for 
CAT(0) cube complexes.

Simple examples also show that their lemma is sharp.  Under a stronger hypothesis, 
\emph{hyperplane-essentiality}, we get more (albeit \emph{using} rank-rigidity):

\begin{propi}\label{propi:SSL}
 Let $X$ be an irreducible, locally finite, essential, hyperplane-essential CAT(0) cube 
complex such 
that $\Aut(X)$ acts cocompactly.  Let $h,v$ be distinct hyperplanes, let $h^+,v^+$ be halfspaces associated to 
$h,v$ 
respectively, and suppose that $h\cap v\neq\emptyset$.  Then $h^+\cap v^+$ contains a hyperplane, and 
therefore contains a halfspace.
\end{propi}

A similar statement appears in~\cite{NevoSageev}. The exact statement 
is~\cite[Proposition~2.11]{BeyrerFioravanti:marked}.  In the latter, the proof is attributed to 
still-in-progress work of the present author and Wilton~\cite{HagenWilton}.  So 
(with our collaborator's blessing), we extracted the proposition and its proof from~\cite{HagenWilton} so 
that an account of Proposition~\ref{propi:SSL} is readily available.  The proof is in 
Section~\ref{sec:strengthened_sector_lemma}.

This seemed important to do because the results of~\cite{BeyrerFioravanti:marked} are significant and use 
the above proposition.  In~\cite{BeyrerFioravanti:marked}, conditions are given under which a cubulation of a 
group $G$ is determined up to equivariant cubical isomorphism by function that assigns to each element of $G$ 
its $\ell_1$ translation length.  In~\cite{BeyrerFioravanti:marked}, the restrictions on the cube complex 
include the 
hypothesis that it has no free faces, and the same result holds when $G$ is hyperbolic under the weaker 
hyperplane-essentiality hypothesis~\cite{BeyrerFioravanti:hyperbolic}.

An essential action on a CAT(0) cube complex is the higher-dimensional version of a minimal action on 
a simplicial tree.  Hyperplane-essentiality requires, in addition, that hyperplane-stabilisers act essentially 
on their hyperplanes.  It always holds in the $1$--dimensional case --- hyperplanes in trees are 
points.  A simple example is that the action of $\integers$ on the tiling of $\reals$ by 
$1$--cubes is hyperplane-essential, but the action of $\integers$ on the CAT(0) cube complex obtained by 
gluing squares corner-to-corner in the obvious way is not.  See Figure~\ref{fig:weird_cubulation}

\begin{figure}[h]
\includegraphics[width=0.5\textwidth]{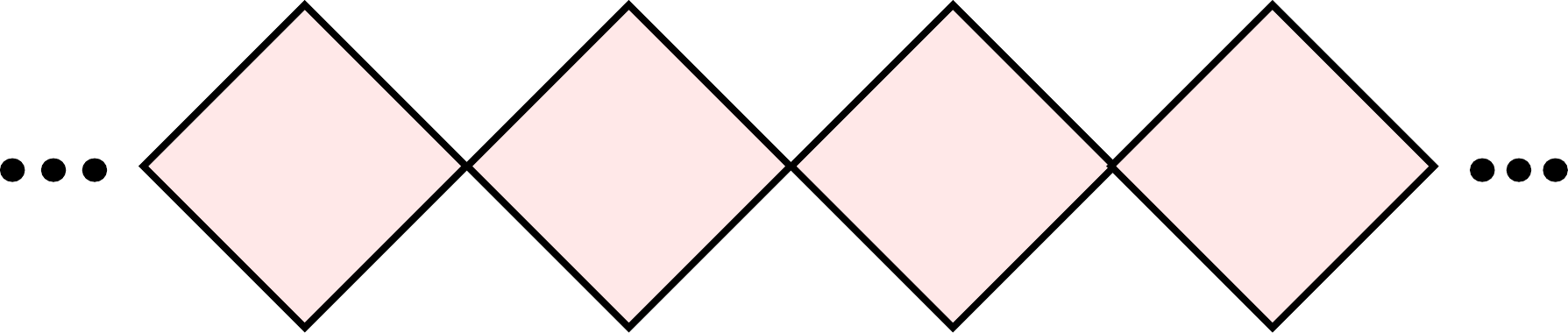}
\caption{The $\integers$--action on this $2$--dimensional CAT(0) cube complex by translations is essential but not hyperplane-essential.  Indeed, the hyperplanes are nontrivial CAT(0) cube 
complexes --- line segments --- with trivial stabilisers, so the action is not hyperplane-essential.  But every halfspace contains points in any fixed $\integers$--orbit arbitrarily far from 
its bounding hyperplane, so the action is essential.}\label{fig:weird_cubulation}
\end{figure}

The combination of essentiality and 
hyperplane-essentiality of a cocompact CAT(0) cube complex $X$ can be viewed as a weak version of $X$ having 
no free faces.  The no free-faces property implies essentiality of hyperplanes of all codimensions. 

Hyperplane-essentiality can often be arranged by modifying the cube complex~\cite{HagenTouikan}, so it is a 
natural hypothesis to impose.  Many motivating examples of actions on CAT(0) cube complexes, like the 
cubulations of hyperbolic $3$--manifold groups constructed by Bergeron-Wise using work of 
Kahn-Markovic~\cite{BergeronWise,KahnMarkovic}, are hyperplane-essential~\cite{BeyrerFioravanti:hyperbolic}.  
In addition to its necessity for length spectrum rigidity, hyperplane-essentiality has recently proved 
useful in other contexts, e.g.~\cite{FioravantiHagen}.  Proposition~\ref{propi:SSL} further illustrates the 
utility of the notion.

The proof of Proposition~\ref{propi:SSL} in this paper relies (in a fairly soft way) on understanding which 
isometries of a locally finite CAT(0) cube complex $X$ act loxodromically on the \emph{contact 
graph}, the $1$--skeleton of the nerve of the covering of $X$ by hyperplane carriers.  Such isometries were 
characterised in~\cite{Hagen:boundary}: a hyperbolic isometry $g$ fails to be loxodromic on the contact 
graph when $g$ is either not rank-one, or some axis of $g$ fellow-travels a hyperplane.  The proof 
in~\cite{Hagen:boundary} relies on \emph{disc diagrams} in CAT(0) cube complexes, which were introduced by 
Casson in unpublished notes, and developed in work of Sageev~\cite{Sageev95} and 
Wise~\cite{Wise:pave}.  

This brings us to the second purpose of this paper, which is expository.  The study of 
rank-one/contracting isometries of CAT(0) spaces/cube complexes is a subject of current interest, see 
e.g.~\cite{CharneySultan,Qing,CordesHume}.  So, in this paper we give a simpler proof of 
the preceding characterisation of contact graph loxodromics, avoiding the use of disc diagrams.  This is 
Theorem~\ref{thm:loxo}.  

One of the ingredients in the proof of Theorem~\ref{thm:loxo} is a description of the convex hull 
of a combinatorial geodesic axis of an isometry (Lemma~\ref{lem:rank_one}).  Convex hulls of geodesics are 
examples of cube complexes in which no three hyperplanes \emph{face}, i.e. for any three disjoint 
hyperplanes, one of them separates the other two.  A naive question is: given a finite set of 
hyperplanes with no such ``facing triple'', is there a geodesic that crosses all hyperplanes in the set?  
Conversely, given a finite set of hyperplanes that cross a ball of a fixed size, can we find large subsets in 
which any three hyperplanes face?

The answer to the first question is no --- the set of hyperplanes in the CAT(0) cube complex formed by arranging five squares cyclically around a common vertex (so that the 
cube complex is homeomorphic to a disc) is a counterexample.  However, when the dimension of the cube 
complex is bounded, there is a geodesic that crosses \emph{a definite proportion} of the hyperplanes; this is 
Corollary~\ref{cor:geodesic} below.  A quantitative version of the second question also has a 
positive answer; see Corollary~\ref{cor:linear_growth}, which says that, if $X$ is a $D$--dimensional CAT(0) 
cube complex that does not isometrically embed in the standard tiling of $\Euclidean^L$ by $L$--cubes for any 
$L$, then for all $N$, there exists $R$ such that the set of hyperplanes intersecting an $R$--ball in $X$ 
contains a facing $N$--tuple of hyperplanes.  

These two results are proved by combining Ramsey's theorem and Dilworth's 
theorem.  The idea to apply Dilworth's theorem to a collection of halfspaces appears throughout the 
literature (see e.g.~\cite{AOS,BCGNW,Fioravanti:dilworth}).  The idea to apply Ramsey's theorem to a 
collection of hyperplanes appears in~\cite{CapraceSageev,ChepoiChalopin} and likely elsewhere.  We are not 
aware of a 
reference where 
the two are applied in concert in this way, so decided to make matters explicit here.    

The main statement on the above topic is Proposition~\ref{prop:ramsey_dilworth}, which says that if $X$ is a 
$D$--dimensional 
CAT(0) cube complex, and $N\in\naturals$, and $\mathcal W$ is a finite set of hyperplanes with no 
$(N+1)$--tuple of facing hyperplanes, then there is a subset of $\mathcal W$ of size at least $|\mathcal W|/K$ 
such that we can choose one halfspace for each hyperplane in the subset in such a way that the associated 
halfspaces are totally ordered by inclusion.  Here, $K$ is a constant depending on $D$ and $N$ (which can be 
made explicit using Ramsey numbers).

The motivation for Corollary~\ref{cor:geodesic} is a question from Abdul Zalloum; the statement seems 
to be useful in current work of Murray-Qing-Zalloum on \emph{sublinearly contracting boundaries}.  The 
purpose of Corollary~\ref{cor:linear_growth} is to illustrate the idea 
of Proposition~\ref{prop:ramsey_dilworth}, which we do by giving a simple proof of the Tits alternative 
for cubulated groups, in Proposition~\ref{prop:tits_1}.

The proof of Proposition~\ref{prop:tits_1} is different from that of the more general statement due to 
Sageev-Wise~\cite{SageevWise:tits}, and more closely resembles the proof in~\cite{CapraceSageev}.  In both 
cases, the main point is to find a facing $4$--tuple of hyperplanes, divide this into two pairs, and apply the 
Double Skewering Lemma to find two hyperbolic isometries which, by ping-pong, generate a free group.

What we find intriguing is that a facing $4$--tuple, if it exists (i.e. if the group in question is not 
virtually abelian), must be seen in a ball in the cube complex of quantifiable radius, because we found it 
using Corollary~\ref{cor:linear_growth}.  In Section~\ref{sec:questions}, we pose some questions 
aimed at effectivising various statements 
about actions on cube complexes.  This is the 
third purpose of this paper.

\textbf{Acknowledgments.}
I thank Elia Fioravanti for asking me about Proposition~\ref{propi:SSL} and for several helpful comments.  I 
thank Henry Wilton for permission to extract 
Proposition~\ref{propi:SSL} 
from~\cite{HagenWilton}.  I thank Talia Fern\`os, Yulan Qing, and Abdul Zalloum for discussions about 
loxodromic directions in the contact graph.  I thank Abdul Zalloum for discussions about facing tuples and 
comments on this paper. 
I thank Radhika Gupta, 
Thomas Ng, Kasia Jankiewicz, Yulan Qing, 
and Abdul Zalloum for a discussion about effective versions of rank rigidity and the Tits alternative during 
the University of Toronto \emph{Hyperbolic Lunch} in April 2020.  I am grateful to the referee for numerous helpful comments.  This work was partly supported 
by EPSRC 
New Investigator Award  EP/R042187/1.

\section{Preliminaries}\label{sec:defn}
There are several introductions to CAT(0) cube complexes, emphasising 
different features; see e.g.~\cite{Chepoi,Haglund:graph_product,Roller,Sageev:book,Wise:book}.  
We follow~\cite[Section 2]{FioravantiHagen}.

\begin{defn}[CAT(0) cube complex]\label{defn:cat0cc}
 A \emph{cube} is a Euclidean unit cube $[-\frac12,\frac12]^n$ for some $n\geq 0$, and a \emph{face} of a cube 
$c$ is a subcomplex obtained by restricting some of the coordinates to $\pm\frac12$.  A \emph{midcube} of $c$ 
is a subspace obtained by restricting exactly one coordinate to $0$.

A \emph{CAT(0) cube complex} is a simply connected CW complex $X$ whose cells are cubes, where the attaching 
maps restrict to isometries on faces, and the following holds.  For each $0$--cube $v\in X$, and each 
collection $e_1,\ldots,e_k$ of $1$--cubes incident to $v$, if for all $i\ne j$ the $1$--cubes $e_i,e_j$ span a 
$2$--cube, then $e_1,\ldots,e_k$ span a unique $k$--cube.  The \emph{dimension} $\dimension X$ is the supremum of the 
dimensions of the cubes of $X$.
\end{defn}

We use the terms ``vertex''/``$0$--cube'', and the terms ``edge''/``$1$--cube'', interchangeably when talking 
about cube complexes.

By~\cite{Gromov:essay,Bridson:thesis,Leary}, a CAT(0) cube complex $X$ supports a CAT(0) 
metric $\dist_2$ in which each (Euclidean) cube is convex.  We will instead work mainly with the path metric 
$\dist$ obtained by equipping each cube with the $\ell_1$ metric; 
see e.g~\cite{Miesch:injective}.  In fact, we mostly work with the restriction of $\dist$ to $X^{(0)}$, 
which is isometric to the metric obtained by restricting to $X^{(0)}$ the usual graph metric on $X^{(1)}$.
%
%

We need the following language about paths in $X$.  A \emph{CAT(0) geodesic} is a geodesic for the 
metric $\dist_2$.  Given $L\in\integers_{\ge0}$, a \emph{combinatorial path} $\gamma:[0,L]\to X$ is a 
continuous map sending $[0,L]\cap\integers$ to $X^{(0)}$ and sending each $[i,i+1]$ isometrically to a 
$1$--cube.  The combinatorial path $\gamma$ is a \emph{combinatorial geodesic} if 
$|i-j|=\dist(\gamma(i),\gamma(j))$ for $0\leq i\leq j\leq L$.

\subsection{Hyperplanes and halfspaces}\label{sec:hyperplanes}
The key features of CAT(0) cube complexes are their \emph{hyperplanes} and \emph{halfspaces}.  There are 
different viewpoints on CAT(0) cube complexes, one emphasising hyperplanes and one emphasising halfspaces. 
 The ubiquity of CAT(0) cube complexes comes from the fact that they are 
``geometric realisations'' of very simple combinatorial data; if one finds the hyperplane viewpoint more 
natural, one thinks of CAT(0) cube complexes as dual to~\emph{wallspaces}, as 
explained in~\cite{Nica,ChatterjiNiblo,HruskaWise}; if one favours halfspaces, one can think of CAT(0) cube 
complexes as dual to \emph{pocsets}, as explained in~\cite{Sageev:book}.  The viewpoints are equivalent, and 
also equivalent to~\emph{discrete median algebras}~\cite{Roller} and~\emph{median graphs}~\cite{Chepoi}.  All 
of these notions are intimately related, and there are different situations making each viewpoint optimal.  

\begin{defn}[Hyperplane]\label{defn:hyperplane}
 Let $X$ be a CAT(0) cube complex.  A \emph{hyperplane} is a connected subspace $h\subset X$ such that for 
each cube $c$ of $X$, the intersection $h\cap c$ is either empty or a midcube of $c$.  The \emph{carrier} 
$\neb(h)$ is the union of all (closed) cubes intersecting $h$.
\end{defn}

Each midcube in $X$ is contained in exactly one hyperplane.

By e.g.~\cite{Sageev95}, if $h$ is a hyperplane, then $h$ is again a CAT(0) cube complex whose 
cubes are midcubes of the form $h\cap c$, where $c$ is a cube of $X$ intersecting $h$.  Moreover, $\neb(h)$ is 
a CAT(0) cube complex isomorphic to $h\times[-\frac12,\frac12]$.  A $1$--cube $e$ with $h\cap e\neq\emptyset$ 
is \emph{dual to} $h$.  The $0$--skeleton of $h$ (regarded as a CAT(0) cube complex) is the set of midpoints 
of $1$--cubes dual to $h$.  The hyperplanes of $h$ are subspaces of the form $a\cap h$, as $a$ varies over the 
hyperplanes of $X$ that intersect $h$.

For each hyperplane $h$, the complement $X-h$ has exactly two components, called \emph{halfspaces} 
\emph{associated} to $h$.  We usually denote these $\OL h$ and $\OR h$.  If $\OL h$ is a halfspace, there is a 
unique hyperplane $h$ such that $\OL h$ is a component of $X-h$, and we also say $h$ is the hyperplane 
\emph{associated} to $\OL h$.

Hyperplanes are not subcomplexes of $X$, but the following construction is often convenient: let $X'$ be the 
CAT(0) cube complex obtained from $X$ by subdividing each $n$--cube into $2^n$ $n$--cubes in the obvious way, 
by declaring the barycentre of each cube to be a $0$--cube.  Then the hyperplanes of $X$ are now 
subcomplexes; they are no longer hyperplanes in the new cube complex $X'$, which has two ``parallel'' copies 
of each of the original hyperplanes.

\begin{defn}[Separation, crossing, parallelism]\label{defn:separation_crossing}
 Given a hyperplane $h$ and $A,B\subset X$, we say that $h$ \emph{separates} $A,B$ if there are 
distinct halfspaces $\OL h,\OR h$ associated to $h$ with $A\subset \OL h$ and $B\subset \OR h$.  We are 
usually interested in the case where $A,B$ are vertices, hyperplanes, or convex subcomplexes (see below).

If $A\subset X$ and the hyperplane $h$ separates two points of $A$, we say that $h$ \emph{crosses} $A$.

The hyperplanes $h,v$ cross if and only if $h,v$ are distinct and have nonempty 
intersection.  Equivalently, each halfspace associated to $h$ intersects each halfspace associated to $v$.

The subcomplexes $A,B$ are \emph{parallel} if any hyperplane $h$ crosses $A$ if and only if $h$ crosses $B$.  
Taking $A,B$ to be $1$--cubes gives an equivalence relation on $1$--cubes in which $1$--cubes 
are equivalent if they are dual to the same hyperplane (i.e. parallel).  So, there is a bijection 
between hyperplanes and parallelism classes of $1$--cubes.
\end{defn}

Crucially, if $\gamma$ is an embedded combinatorial path in $X$, 
then $\gamma$ is a geodesic if and only if $\gamma$ contains at most one $1$--cube dual to each hyperplane.  
In particular, if $x,y$ are vertices, then $\dist(x,y)$ is the number of hyperplanes separating $x$ from $y$.

\begin{rem}[Helly property for hyperplanes]
 Let $h_1,\ldots,h_n$ be hyperplanes that pairwise cross.  Then $\bigcap_{i=1}^nh_i\neq\emptyset$, i.e. there 
is a (not necessarily unique) $n$--cube whose barycentre is contained in each $h_i$.  So $\dimension X$ is the 
maximum possible cardinality of a set of pairwise crossing hyperplanes.

The same \emph{Helly property} --- any finite collection of pairwise intersecting subsets in a given class has 
nonempty total 
intersection --- also holds for the class of convex subcomplexes, discussed presently.
\end{rem}

\subsection{Medians, convexity, geodesics, and gates}\label{sec:medians}
In~\cite{Chepoi}, Chepoi established a correspondence between \emph{median graphs} and CAT(0) cube 
complexes.  Let $\Gamma$ be a connected graph, with graph metric $\rho$.  The \emph{interval} $[a,b]$ between 
vertices $a,b$ is the set of vertices $c$ with $\rho(a,b)=\rho(a,c)+\rho(b,c)$.  The graph $\Gamma$ is 
\emph{median} if for all $a,b,c\in \Gamma^{(0)}$, there is a unique vertex $\mu(a,b,c)$ with 
$\mu(a,b,c)\in[a,b]\cap[b,c]\cap[a,c]$.  Chepoi's theorem says that the $1$--skeleton of any CAT(0) cube 
complex is a median graph, and each median graph is the $1$--skeleton of a uniquely determined 
CAT(0) cube complex.

Given a CAT(0) cube complex $X$, we let $\mu:(X^{(0)})^3\to X^{(0)}$ be the median operator described above.  
There is a way to extend $\mu$ over the whole of $X$, but we will not need it here.

The subcomplex $Y$ of $X$ is \emph{full} if $Y$ contains every cube of $X$ whose $0$--skeleton appears in $Y$. 
 The full subcomplex $Y$ is \emph{convex} if for all vertices $x,y\in Y$ and $z\in X$, we have $\mu(x,y,z)\in 
Y$.  If $x,y$ are vertices, then the median interval $[x,y]$ is convex and consists of the 
union of combinatorial geodesics from $x$ to $y$.  If $Y$ is a convex subcomplex, then $[x,y]\subset Y$ 
for all $x,y\in Y^{(0)}$, and conversely.

If $\OL h$ is a halfspace, then the smallest subcomplex containing $\OL h$ is convex.  If $h$ is a hyperplane, 
then the carrier $\neb(h)$ is convex.    

Given a subspace $A$ of $X$, the \emph{convex hull} of $A$ is defined as follows.  First, let $A'$ be the 
intersection of all halfspaces containing $A$.  Then the convex hull is the union of all cubes contained in 
$A'$.  Convex hulls are convex; the convex hull of a pair of vertices $x,y$ is exactly the median 
interval $[x,y]$.

The median viewpoint enables a very useful construction, the \emph{gate map}.  Let $Y\subset X$ be a convex 
subcomplex.  Then there is a map $\gate=\gate_Y:X\to Y$ with the following properties (see e.g.~\cite[Section 
3]{BHS:HHS_I}:
\begin{itemize}
 \item $\gate$ is $1$--lipschitz (for both $\dist$ and $\dist_2$);
 \item if $x\in X^{(0)}$ and $h$ is a hyperplane, then $h$ separates $x$ from $\gate(x)$ if and only if $h$ 
separates $x$ from $Y$;
\item $\dist(x,\gate(x))=\dist(x,Y)$, and $\gate(x)$ is the unique closest vertex of $Y$ to $x$.
\end{itemize}

If $Y,Z$ are convex subcomplexes, then $\gate_Y(Z)$ and $\gate_Z(Y)$ are isomorphic CAT(0) cube complexes, and 
a hyperplane $h$ crosses $\gate_Y(Z)$ if and only if $h$ crosses both $Y$ and $Z$.  Moreover, there is a 
convex subcomplex $\gate_Y(Z)\times I$ of $X$, where the hyperplanes crossing $I$ are precisely those that 
separate $Y$ from $Z$.  In fact, $\gate_Y(Z)\times I$ is the convex hull of $\gate_Z(Y)\cup \gate_Y(Z)$.  We 
will use this in the proof of Theorem~\ref{thm:loxo}, when we note that 
$\diam(\gate_Y(Z))=\diam(\gate_Z(Y))$.  (See e.g.~\cite[Lemma 2.6]{BHS:HHS_I}.)

The following lemma is standard, follows easily from the above itemised properties (specifically the second) 
and we will use it later.  It appears in various places in the literature; see e.g.~\cite[Proposition 
2.6]{Genevois:diagram}.

\begin{lem}\label{lem:gate_sep}
 Let $X$ be a CAT(0) cube complex and $Y$ a convex subcomplex.  Let $\gate:X\to Y$ be the gate map.  Let 
$x,y\in X$ be $0$--cubes.  Then for all hyperplanes $h$, we have that $h$ separates $\gate(x),\gate(y)$ if and 
only if $h$ both intersects $Y$ and separates $x,y$.
\end{lem}

Although a hyperplane $h$ is not a subcomplex, we saw above that $h$ becomes a subcomplex in the cubical 
subdivision $X'$, and moreover it is convex.  Accordingly, we also have a gate map $\gate_h:X\to h$.  We will 
also use this in Theorem~\ref{thm:loxo} and Proposition~\ref{propi:SSL}.  More on gate maps to hyperplanes can 
be found in~\cite[Section 2.1.6]{FioravantiHagen}.

\begin{rem}
 If $Y\subset X$ is a subcomplex, then convexity of $Y$ in the above sense follows from convexity of $Y$ 
in the CAT(0) metric $\dist_2$~\cite{HaglundSemisimple}.  However, this only works for subcomplexes.  For general subspaces, convexity in the above sense implies CAT(0) convexity, but the converse 
might not hold: consider a diagonal line in the standard tiling of $\Euclidean^2$ by $2$--cubes.
\end{rem}

\subsection{Isometries and skewering}\label{sec:isometries}
In this subsection, we mostly follow~\cite{HaglundSemisimple} and~\cite{CapraceSageev}.

By $\Aut(X)$, we mean the group of cubical automorphisms of the CAT(0) cube complex $X$.  Automorphisms 
are isometries with respect to $\dist$ and $\dist_2$, although $(X,\dist_2)$ may have isometries that are not 
cubical (this is studied in~\cite{Bregman}).

The action $G\to\Aut(X)$ of the group $G$ is \emph{proper} if cube stabilisers are finite, and 
\emph{metrically proper} if for all $x_0\in X^{(0)}$ and all $R\ge0$, the set of $g\in G$ with 
$\dist(x_0,gx_0)\leq R$ is finite.  If $X$ is locally finite, then any proper action is metrically proper.

The action is \emph{cocompact} if there is a compact subcomplex $K\subset X$ with $G\cdot K=X$.

The following is well-known and widely-used; see e.g.~\cite[Lemma 2.3]{FioravantiHagen} for a proof:

\begin{lem}[Hereditary cocompactness]\label{lem:hereditary_cocompactness}
 Let $X$ be a CAT(0) cube complex on which the group $G$ acts cocompactly.  Then for all hyperplanes $h$ of 
$X$, the action of $\stabilizer_G(h)$ on $h$ is cocompact.
\end{lem}

The cube complex $X$ is \emph{essential} if each halfspace $\OL h,\OR h$ contains points arbitrarily far from 
the 
associated hyperplane $h$.  The action of $G$ on $X$ is \emph{essential} if those points can all be chosen in 
a fixed $G$--orbit.  When $\Aut(X)$ acts on $X$ cocompactly, $X$ is essential if and only if the action of 
$\Aut(X)$ is essential.  

The cube complex $X$ is \emph{hyperplane-essential} if each hyperplane $h$ of $X$ is an essential CAT(0) cube 
complex, and $G\to\Aut(X)$ is a \emph{hyperplane-essential} action if, for each hyperplane $h$, the action of 
$\stabilizer_G(h)$ on $h$ is essential.

Given $g\in\Aut(X)$, we say that $g$ is \emph{combinatorially hyperbolic} if there is a combinatorial geodesic 
$\gamma:\reals\to X$ and a positive integer $\ell$ such that $g\gamma(t)=\gamma(t+\ell)$ for all $t\in\reals$, 
i.e. $g$ acts on $\gamma$ as a nontrivial translation.  Such a $\gamma$ is a \emph{combinatorial axis} for 
$g$.

If $g$ is combinatorially hyperbolic, then the set of hyperplanes $h$ such that $h$ intersects an axis of $g$ 
is independent of the choice of axis.  If $h$ is a hyperplane, then $g$ \emph{skewers} $h$ if $g\OL 
h\subset \OL h$, where $\OL h$ is one of the halfspaces associated to $h$.

Here is an exercise: if the convex hull $A$ of the axis of $g$ is finite-dimensional, then any 
hyperplane crossing $A$ is skewered by some power of $g$.

A theorem of Haglund~\cite{HaglundSemisimple} asserts that, under a mild assumption on $X$ that can always be 
arranged by replacing $X$ by its cubical subdivision, any $g\in\Aut(X)$ is either combinatorially hyperbolic 
or fixes a vertex.  Even without subdividing, any $g$ has a positive power that is either combinatorially 
hyperbolic or fixes a vertex, provided $X$ is finite-dimensional.

\begin{rem}
Higher-dimensional versions of Haglund's combinatorial semisimplicity theorem, and related results, have been 
obtained by Woodhouse, Genevois, and Woodhouse-Wise~\cite{Genevois:axis,Woodhouse:axis,WoodhouseWise}.
\end{rem}

For the CAT(0) metric, one can deduce the following (see e.g.~\cite[Exercise II.6.6(2)]{BH}):

\begin{lem}\label{lem:CAT0hyperbolic}
 Let $X$ be a finite-dimensional CAT(0) cube complex and let $g\in\Aut(X)$.  Then either $g$ fixes a point in 
$X$, or $g$ is a hyperbolic isometry of the CAT(0) space $(X,\dist_2)$.
\end{lem}

The finite dimensional hypothesis is necessary: there are infinite dimensional CAT(0) cube complexes 
with isometries that are combinatorially hyperbolic but CAT(0) parabolic~\cite{AlgomKfir}.

\begin{defn}[Rank one]\label{defn:rank_one}
If $g\in\Aut(X)$ is hyperbolic for the CAT(0) metric, we say that $g$ is \emph{not rank one} if some CAT(0) 
geodesic axis for $g$ lies in an isometrically embedded Euclidean half-flat $[0,\infty)\times\reals$, and $g$ 
is \emph{rank one} otherwise.
\end{defn}

The Double Skewering Lemma of Caprace-Sageev~\cite[p. 4]{CapraceSageev} is a vital tool for identifying 
hyperbolic isometries of CAT(0) cube complexes, given an ambient essential action.

\begin{lem}[Double skewering]\label{lem:double_skewering_lemma}
Let $X$ be a finite-dimensional CAT(0) cube complex and let $G$ act essentially on $X$.  Suppose that one of 
the following holds:
\begin{itemize}
 \item $G$ acts with no fixed point in the visual boundary $\boundary X$;
 \item $X$ is locally finite and $G$ acts cocompactly.
\end{itemize}
Let $h,v$ be disjoint hyperplanes, and let $\OL h,\OL v$ be halfspaces associated to $h,v$ respectively, with 
$\OL h\subsetneq\OL v$.  Then there exists a hyperbolic (in the combinatorial and CAT(0) metrics) element 
$g\in G$ such that $g\OL v\subsetneq\OL h$; in particular, $h$ separates $v$ from $gv$.
\end{lem}

\begin{proof}
 The ``in particular'' statement follows immediately from $g\OL v\subsetneq\OL h$.  Second, it also follows 
immediately from $g\OL v\subsetneq\OL h$ that $\langle g\rangle$ has unbounded orbits in $X$.  Hence, 
up to replacing $g$ by a positive power, $g$ is combinatorially hyperbolic.  Since $X$ is finite-dimensional, 
the 
identity $(X,\dist)\to (X,\dist_2)$ is a $G$--equivariant quasi-isometry (by Lemma~\ref{lem:QI} below), so 
$\dist_2(g^nx,x)$ grows linearly 
in $n$, for any $x\in X$, whence $g$ is also hyperbolic in the CAT(0) metric.

So, it remains to find $g$ with $g\OL v\subsetneq\OL h$.  In the case where $G$ acts without a fixed point at 
infinity, the desired statement appears on page 4 of~\cite{CapraceSageev}.

Suppose $G$ acts cocompactly on $X$ and $X$ is locally finite. 
 In this case, Corollary~4.9 of~\cite{CapraceSageev} implies that $X=X_1\times\cdots\times X_p\times Y$, where 
each $X_i$ has compact hyperplanes, and the finite-index subgroup $G'\leq G$ preserving this decomposition 
does not fix a point in $\boundary Y$.

If $h,v$ are hyperplanes of the form $X_1\times\cdots\times X_p\times \bar h,X_1\times\cdots\times X_p\times 
\bar v$, where $\bar h,\bar v$ are hyperplanes of $Y$, then the claim follows from the version for actions 
without fixed points at infinity.  So, it suffices to prove the claim in the case where hyperplanes of $X$ are 
compact, $X$ is locally finite, and $G$ acts cocompactly and essentially; we leave this as an exercise.
\end{proof}

\subsection{The contact graph}\label{sec:contact_graph}
Let $X$ be a CAT(0) cube complex and let $\mathcal W$ be the set of hyperplanes.  Then $\{\neb(h):h\in\mathcal 
W\}$ is a covering of $X$, and we define the \emph{contact graph} $\contact X$ to be the 
(necessarily connected) $1$--skeleton of the 
nerve of this covering, i.e. the intersection graph of the hyperplane carriers.  This graph was initially 
defined in~\cite{Hagen:contact}; it has a vertex for each hyperplane, with two vertices adjacent provided no 
third hyperplane separates the corresponding hyperplanes.  By~\cite[Theorem 4.1]{Hagen:contact}, $\contact X$, 
equipped with its usual 
graph metric, is quasi-isometric to a tree (with constants independent of $X$).  In particular, $\contact X$ 
is hyperbolic.

Throughout the paper, if $h$ is a hyperplane of $X$, we also use the letter $h$ to mean the corresponding 
vertex of $\contact X$.

We now define a coarse map $\pi:X\to 2^{\contact X}$ by first defining $\pi$ on vertices of $X$, and then on higher-dimensional open cubes.

Given $x\in X^{(0)}$, the set of hyperplanes $h$ with $x\in\neb(h)$ corresponds to a complete subgraph of 
$\contact X$, which we denote $\pi(x)$.  Hence $\pi:X^{(0)}\to 2^{\contact X}$ is a coarse map.  For each open edge 
$e$ of $X$, we define $\pi(e)$ to be the vertex corresponding to the hyperplane dual to $e$.  Hence we have a 
coarse map $\pi:X^{(1)}\to 2^{\contact X}$.  More generally, if $c$ is an open $n$--cube, $n\geq1$, let 
$\pi(c)$ be the complete subgraph with vertex set the hyperplanes intersecting $c$.  
%

So, $\pi:X\to 2^{\contact X}$ sends cubes to bounded sets, and, if $f,c$ are open cubes with $\bar 
f\subset\bar c$, then $\pi(f)\subset\pi(c)$.  If $x$ is a vertex of $\bar f$, then 
$\pi(x)\cap\pi(f)\neq\emptyset$.  (Here, $\bar f,\bar c$ denote the closures of $f,c$.)  In particular, 
$\pi:X^{(1)}\to 2^{\contact X}$ is coarsely lipschitz.

Since $\Aut(X)$ acts on the set of hyperplanes of $X$, preserving intersection and non-intersection of 
carriers, we get a homomorphism $\Aut(X)\to\Aut(\contact X)$; the relationship between these two groups is 
studied in~\cite{Fioravanti:contact}.

With respect to this action of $\Aut(X)$ on $\contact X$, the map $\pi$ is equivariant, i.e. $\pi(gx)=g\pi(x)$ 
for all $x\in X,g\in\Aut(X)$.
%

We will chiefly be interested in when an element $g\in\Aut(X)$ acts on $\contact X$ \emph{loxodromically}, 
i.e. when the map $\integers\to\integers$ given by $n\mapsto \dist_{\contact X}(h,g^nh)$ is bounded above and 
below by increasing linear functions of $n$ (the functions depend on $h$, but the existence of such functions 
for some hyperplane implies it for any other hyperplane).

Since $\pi$ is equivariant and coarsely lipschitz, to prove that a hyperbolic isometry $g\in\Aut(X)$ is 
loxodromic, it suffices to prove the following: for any hyperplane $h$, there exists $C>0$ such that 
$\dist_{\contact X}(h,g^nh)>Cn$ for all $n>0$.  A non-hyperbolic isometry can never be loxodromic on $\contact 
X$, since it stabilises a clique $\pi(x)$, where $x\in X$ is a point fixed by $g$.

Finally, here is a simple observation that is used in the proof of Proposition~\ref{propi:SSL}:

\begin{lem}\label{lem:single_point}
Let $X$ be a CAT(0) cube complex and let $v,h$ be hyperplanes such that $\dist_{\contact X}(v,h)>2$.  Then 
$\gate_v(h)$ is a single point.
\end{lem}

\begin{proof}
Recall that, when $v$ is regarded as a CAT(0) cube complex, $\gate_v(h)$ is a convex subcomplex of $v$ whose 
hyperplanes have the form $a\cap v$, where $a$ is a hyperplane of $X$ that crosses both $v$ and $h$.  If 
$\dist_{\contact X}(v,h)>2$, there are no such hyperplanes $a$, so $\gate_v(h)$ is a CAT(0) cube complex with 
no hyperplanes and hence no positive-dimensional cubes, i.e. $\gate_v(h)$ is a point. 
\end{proof}

\section{Facing tuples and chains with Dilworth and Ramsey}\label{sec:ramsey_dilworth_tricks}
Fix a CAT(0) cube complex $X$.   Recall that $\dimension X$ is equal to the supremum of the cardinalities of 
sets of pairwise intersecting hyperplanes.

\begin{defn}[Facing tuple]\label{defn:facing_tuple}
 Let $n\in\naturals\cup\{\infty\}$.  A \emph{facing $n$--tuple} is a set of hyperplanes $\{h_1,\ldots,h_n\}$ 
with the property that, for each $h_i$, we can choose an associated halfspace $\OL h_i$ such that 
$\OL h_i\cap\OL h_j=\emptyset$ for $i\neq j$.  Equivalently, there do not exist $i,j,k$ such that $h_i$ 
separates $h_j$ from $h_k$.
\end{defn}

\begin{defn}[Chain]\label{defn:chain}
 A \emph{chain} in $X$ of length $n$ is a set $\{h_1,\ldots,h_n\}$ of hyperplanes such that $h_i$ 
separates $h_{i-1}$ from $h_{i+1}$ for $2\leq i\leq n-1$.
\end{defn}

Let $h_1,\ldots,h_n$ be a chain. For $2\leq i\leq n$, let $\OL h_i$ be the halfspace associated to $h_i$ 
and containing $h_1$, and let $\OL h_1$ be the halfspace associated to $h_1$ not containing $h_2$.  Then $\OL 
h_1\subsetneq\cdots\subsetneq\OL h_n$.  Conversely, let $h_1,\ldots,h_n$ be hyperplanes for which we can 
choose a 
halfspace $\OL h_i$ associated to each $h_i$ in such a way that $\OL h_1\subsetneq\cdots\subsetneq \OL h_n$.  
Then $\{h_1,\ldots,h_n\}$ is a chain.

The following is a useful trick:

\begin{prop}[Chains and facing tuples]\label{prop:ramsey_dilworth}
For all $D,N\in\naturals$, there exists $K(D,N)\geq 1$ such that the following holds.  Let $X$ be a 
$D$--dimensional CAT(0) cube complex, and let $\mathcal W$ be a finite set of hyperplanes in $X$ such that any 
facing tuple in $\mathcal W$ has cardinality at most $N$.  Then $\mathcal W$ contains a chain of cardinality 
at least $|\mathcal W|/K(D,N)$.
\end{prop}

\begin{proof}

Let $\widehat{\mathcal W}=\{\OL h:h\in\mathcal W\}$ be a set of halfspaces with exactly one 
associated to each hyperplane in $\mathcal W$.  The set $\widehat{\mathcal W}$ is partially ordered by 
inclusion.  For each $h\in\mathcal W$, recall the notation $\OR h=X-\OL h$.

\textbf{Bounding $\subseteq$--antichains with Ramsey's theorem:}  Let $\mathcal A\subset\widehat{\mathcal W}$ 
be a set of halfspaces, no two of which are $\subseteq$--comparable.  So, for all $\OL h,\OL v\in\mathcal A$, 
exactly one of the following holds:
\begin{enumerate}
 \item \label{item:blue} The hyperplanes $h,v$ are distinct and $h\cap v\neq \emptyset$.
 \item \label{item:red} The hyperplanes $h,v$ are disjoint. So, one of the following holds: 
$\OL v\cap\OL h=\emptyset$, or $\OR v\cap \OR h=\emptyset$.
\end{enumerate}

Suppose that $\OL h_1,\ldots,\OL h_n\in\mathcal A$ are distinct halfspaces with the property that 
item~\eqref{item:blue} holds for $\OL h_i,\OL h_j$ for all $i\neq j$.  Then the set $\{h_1,\ldots,h_n\}$ 
contains $n$ distinct, pairwise-intersecting hyperplanes.  Hence $n\leq D$.  


Suppose that $\OL h_1,\ldots,\OL h_n\in\mathcal A$ are distinct halfspaces with the property that 
item~\eqref{item:red} holds for $\OL h_i,\OL h_j$ for all $i\neq j$.  

For all $i,j,k\leq n$, the hyperplane $h_i$ cannot separate $h_j$ from $h_k$.  Indeed, suppose that 
this is the case.  Then, up to relabelling, $h_j\subset \OL h_i$.  Since $\OL h_i$ and $\OL h_j$ are 
$\subseteq$--incomparable, $h_i\subseteq\OL h_j$, i.e. $\OR h_i\cap\OR h_j=\emptyset$.  But then either $\OL 
h_k$ contains $h_i,h_j$ and hence contains $\OL h_i$, or $\OL h_k$ is contained in $\OL h_j$.  Either of these 
situations violates pairwise-incomparability of the elements of $\mathcal A$.  

We have just shown that $\{h_1,\ldots,h_n\}$ form a facing tuple.  Hence $n\leq N$.

%
%
%

Let $G$ be the complete graph with vertex set $\mathcal A$.  We colour an edge blue if the corresponding pair 
of halfspaces satisfy item~\eqref{item:blue} and red if the halfspaces satisfy item~\eqref{item:red}.  This 
colours all of the edges.  We have shown that blue cliques have at most $D$ vertices, and red cliques have at 
most $N$ vertices.  So, by \textbf{Ramsey's theorem}~\cite{Ramsey}, $|\mathcal A|\leq 
\mathrm{Ram}(D+1,N+1)-1$, where $\mathrm{Ram}(\bullet,\bullet)$ denotes the Ramsey number.

\textbf{Applying Dilworth's theorem:}  Let $K(D,N)=\mathrm{Ram}(D+1,N+1)-1$.  We have shown that 
antichains in $\widehat{\mathcal W}$ have cardinality at most $K(D,N)$.  So, by \textbf{Dilworth's 
theorem}~\cite{Dilworth}, we have a partition $$\widehat{\mathcal 
W}=\bigsqcup_{i=1}^{C}\widehat{\mathcal W}_i,$$ where $C\leq K(D,N)$ and each $\widehat{\mathcal W}_i$ is a set 
of halfspaces that is totally ordered by inclusion.  Let $\mathcal W_i$ be the set of hyperplanes $h$ such that 
the halfspace associated to $h$ and belonging to $\widehat{\mathcal W}$ appears in $\widehat{\mathcal W}_i$.  
Then $\mathcal W_i$ is a chain in the 
sense of Definition~\ref{defn:chain}, and for some $i$, we have $|\mathcal W_i|\geq|\mathcal W|/K(D,N)$.
\end{proof}

Here are some consequences.  The first answers a question posed by Abdul Zalloum:

\begin{cor}\label{cor:geodesic}
 Let $X$ be a CAT(0) cube complex of dimension $D<\infty$.  Let $\mathcal W$ be a set of hyperplanes in $X$ 
that does not contain a facing triple.  Then $X$ contains a (combinatorial or CAT(0)) geodesic $\gamma$ such 
that $\gamma$ intersects at least $|\mathcal W|/K(D,2)$ of the hyperplanes in $\mathcal W$.
\end{cor}

\begin{proof}
 Use Proposition~\ref{prop:ramsey_dilworth}, with $N=2$, to find a chain $\mathcal C=\{h_1,\ldots,h_n\}$ of 
cardinality at least $|\mathcal W|/K(D,2)$ in $\mathcal W$.  Choose associated halfspaces $\OL h_1,\ldots\OL 
h_n$ with $\OL h_1\subset\cdots\subset\OL h_n$.  Choose vertices $x\in\OL h_1$ and $y\in\OR h_n$, and let 
$\gamma$ be any geodesic from $x$ to $y$.
\end{proof}

In a similar vein, we can slightly strengthen Lemma 2.1 from~\cite{CapraceSageev}. This was pointed out by 
Elia Fioravanti and Abdul Zalloum. 

\begin{cor}\label{cor:pencil}
 Let $X$ be a CAT(0) cube complex of dimension $D<\infty$.  Let $k\in\naturals$.  Let $\gamma$ be a geodesic 
(in the combinatorial or CAT(0) metric) that crosses at least $D\cdot k$ hyperplanes.  Then the set of 
hyperplanes crossing $\gamma$ contains a chain of cardinality $k$.
\end{cor}

\begin{proof}
 Let $\mathcal W$ be the set of hyperplanes intersecting $\gamma$ and note that $\mathcal W$ cannot contain a 
facing triple.  Applying Proposition~\ref{prop:ramsey_dilworth} would yield the desired statement with $Dk$ 
replaced by $K(D,2)k$.  However, one can do a bit better: for each $h\in\mathcal W$, let $\OL h$ be the 
halfspace associated to $h$ and containing $\gamma(0)$.  Given $h,v\in\mathcal W$, the halfspaces $\OL h,\OL v$ 
are $\subset$--incomparable only if $h,v$ cross, so the claim follows by applying Dilworth's theorem to $\{\OL 
h:h\in\mathcal W\}$.
\end{proof}

It is well known that, when $X$ is finite-dimensional, the CAT(0) metric is quasi-isometric to the 
combinatorial metric (see e.g.~\cite[Lemma 2.2]{CapraceSageev}).  In the interest of self-containment, here we give a cosmetically different 
statement and proof~\cite{CapraceSageev} working with arbitrary points instead of vertices:

\begin{lem}\label{lem:QI}
 Let $X$ be a CAT(0) cube complex of dimension $D<\infty$.  Then there exist $\lambda_0\ge1,\lambda_1\ge0$, 
depending only on $D$, such that the following holds.  Let $x,y\in X$ be arbitrary points, and let $\mathcal 
W(x,y)$ be the set of hyperplanes separating $x$ from $y$.  Then 
$$\frac{1}{\lambda_0}\dist_2(x,y)-\lambda_1\leq |\mathcal W(x,y)|\leq \lambda_0\dist_2(x,y)+\lambda_1.$$
\end{lem}

\begin{proof}
 Let $\gamma$ be a CAT(0) geodesic joining $x$ to $y$.  Note that a hyperplane $h$ intersects $\gamma$ in a 
single point if and only if $h\in\mathcal W(x,y)$ (any other hyperplane either contains $\gamma$ or is 
disjoint from $\gamma$).  By Corollary~\ref{cor:pencil}, $\mathcal W(x,y)$ contains a chain $\mathcal C$ of 
hyperplanes with cardinality at least $\lfloor|\mathcal W(x,k)|/D\rfloor$.  If $h,v\in\mathcal C$ are 
distinct, then $\dist_2(h,v)\geq 1$, so $|\gamma|\geq |\mathcal C|-1$.  Hence $$|\mathcal W(x,y)|\leq 
D\dist_2(x,y)+2D.$$

To prove the other inequality, let $c_x,c_y$ be cubes containing $x,y$ respectively.  Then the set $\mathcal 
W(c_x,c_y)$ of hyperplanes separating $c_x,c_y$ satisfies $|\mathcal W(c_x,c_y)|\leq|\mathcal 
W(x,y)|$.  Now, fix a combinatorial geodesic $\alpha:[0,L]\to X$ from $c_x$ to $c_y$ having length $|\mathcal 
W(c_x,c_y)|$.  Then $$\dist_2(x,y)\leq 
\dist_2(x,\alpha(0))+\dist_2(y,\alpha(L))+\dist_2(\alpha(0),\alpha(L)).$$  Just because $d_2$ is a 
path-metric, and edges of $X$ have $d_2$--length $1$, 
$$\dist_2(\alpha(0),\alpha(L))\leq\dist(\alpha(0),\alpha(L))=|\mathcal W(c_x,c_y)|.$$  Since $x,\alpha(0)$ lie 
in a common cube $c_x$, we have $\dist_2(x,\alpha(0)),\dist_2(y,\alpha(L))\leq \sqrt{D}$.  So 
$$\dist_2(x,y)-2\sqrt{D}\leq 
|\mathcal W(x,y)|,$$ as required.
\end{proof}

We can also produce large facing tuples, given a bound on chains.

\begin{defn}\label{defn:H_r}
 Let $X$ be a CAT(0) cube complex, let $x_0\in X^{(0)}$, and let $B_R(x_0)$ be the set of vertices $y\in X$ 
with $\dist(x_0,y)\leq R$.  Then $\mathcal H_R$ denotes the set of hyperplanes $h$ such that $h$ crosses 
$B_R(x_0)$.
\end{defn}

Observe:

\begin{lem}\label{lem:chain_H_R}
 Any chain in $\mathcal H_R$ has cardinality at most $2R$.
\end{lem}

Now we can use Proposition~\ref{prop:ramsey_dilworth}:

\begin{cor}\label{cor:linear_growth}
 Let $X$ be a CAT(0) cube complex with dimension $D<\infty$, let $x_0\in X^{(0)}$, and for each $R$, let 
$\mathcal H_R$ be as in Definition~\ref{defn:H_r}. Let $N\in\naturals$.  Then one of the following holds:
\begin{itemize}
 \item There exists $R_0$ such that for all $R\geq R_0$, there is a facing $(N+1)$--tuple in $\mathcal H_R$.
 \item $(X,\dist)$ admits an isometric embedding in the standard tiling of $\Euclidean^L$ by $L$--cubes, where 
$L\leq K(D,N)$.
\end{itemize}
In the second case, $|\mathcal H_R|$ grows at most linearly in $R$.
\end{cor}

\begin{proof}
Suppose that the first conclusion fails.  Then $N$ bounds the cardinality of facing tuples in 
$\mathcal H_R$ for all $R\geq 0$.  Let $\widehat{\mathcal H_R}$ be a set of halfspaces associated to 
hyperplanes in $\mathcal H_R$, with one halfspace per hyperplane.  By the proof of 
Proposition~\ref{prop:ramsey_dilworth}, 
$\widehat{\mathcal H_R}$ can be partitioned into $L\leq K(D,N)$ sets 
$\widehat{H_R}^1,\ldots,\widehat{H_R}^L$, each totally ordered by inclusion.  

Hence $\mathcal H_R$ can be partitioned into $L$ chains $\mathcal H_R^1,\ldots,\mathcal H_R^L$.  For $i\leq 
R$, let $f_i:X\to X_i$ be the \emph{restriction quotient} (see~\cite[Section 2]{CapraceSageev}) obtained by 
cubulating the wallspace $(X^{(0)},\mathcal H_R^i)$.  Since $\mathcal H_R^i$ is a finite chain, $X_i$ is 
isomorphic to the tiling of a finite segment by $1$--cubes.

Taking the product gives a map $f':X\to \prod_iX_i$ which is an isometric embedding on $B_R(x_0)$.  Since 
$\diam(X_i)\leq 2R$, by Lemma~\ref{lem:chain_H_R}, we can isometrically embed $X_i$ in the cubical tiling of 
$[-R,R]$, take the product over $i$ of these embeddings, and compose with $f'$ to get a (combinatorial) 
isometric embedding $f_R:B_R(x_0)\to [-R,R]^L$.

For each $R$, there are finitely many such embeddings, and for $R=0$, the embedding is unique.  Let $G$ be the 
graph whose vertex set is the set of isometric embeddings $f_R:B_R(x_0)\to[-R,R]^L$ for $R\ge 0$.  Join 
vertices $f_R$ and $f_{R+1}$ by an edge if $f_{R+1}$ restricts to $f_R$ on $B_R(x_0)$.  Then $G$ is a 
connected, locally finite graph.  So, by K\"onig's lemma~\cite{Konig}, we 
obtain a combinatorial isometric embedding of $X$ in $\Euclidean^L$, as required.

Finally, in this situation, Lemma~\ref{lem:chain_H_R} and Proposition~\ref{prop:ramsey_dilworth} combine to 
prove that $|\mathcal H_R|\leq 2RK(D,N)$ for all $R$.
\end{proof}

Using the above, we recover a Tits alternative for cubulated groups; see 
also~\cite{CapraceSageev,SageevWise:tits}.

\begin{prop}\label{prop:tits_1}
 Let the group $G$ act essentially on the $D$--dimensional CAT(0) cube complex $X$.  Suppose 
that the action of $G$ is either cocompact and $X$ is locally finite, or $G$ has no global fixed point in 
$\boundary X$.  Then either $G$ contains a nonabelian free group, or $X$ admits a combinatorial isometric 
embedding in $\Euclidean^L$, where $L\leq K(D,3)$.
\end{prop}

\begin{proof}
 By Corollary~\ref{cor:linear_growth}, either the second conclusion holds, or there exists $R$ such that 
$\mathcal H_R$ contains a facing $4$--tuple $a,b,c,d$. Let $\OL a$ be the 
halfspace associated to $a$ that is disjoint from $b,c,d$, and define $\OL b,\OL c,\OL d$ analogously.  Let 
$\OR a=X-\OL a$, and define $\OR b,\OR c,\OR d$ analogously.

Apply the Double Skewering Lemma to find $g,h\in G$ such that $g\OR b\subsetneq\OL a$ and $h\OR d\subsetneq\OL 
c$.  

Hence $g(X-\OL b)\subset \OL a$ and $g^{-1}(X-\OL a)\subset\OL b$.  Similarly, $h(X-\OL d)\subset\OL c$ and $h^{-1}(X-\OL c)\subset \OL b$.  So, applying the ping-pong lemma to the four disjoint sets 
$\OL a,\OL b,\OL c,\OL d$ and the elements $g,h$ shows that $\langle g,h\rangle\cong F_2$.
%
%
%
\end{proof}

\begin{cor}\label{cor:tits_alternative}
 Let the group $G$ act properly and cocompactly on the $D$-dimensional CAT(0) cube complex $X$.  Then $G$ 
contains a nonabelian free group or $G$ is virtually finite-rank abelian.
\end{cor}

\begin{proof}
Using~\cite[Proposition 3.5]{CapraceSageev}, we can assume that the 
action of $G$ on $X$ is essential.  By Proposition~\ref{prop:tits_1}, either $G$ contains a nonabelian free 
group, or $X$ isometrically embeds in $\Euclidean^L$ for some $L$.  Observe that $G$ is finitely generated, and 
the composition of an orbit map $G\to X$ with the embedding $X\to\Euclidean^L$ shows that $G$ has polynomial 
growth.  Thus $G$ is virtually nilpotent~\cite{Gromov:polynomial} and hence virtually 
abelian~\cite[Theorem II.7.8]{BH}.
\end{proof}

\begin{rem}
The above proof is similar 
to the more general proof in~\cite{CapraceSageev}, but focused on what we hope could be a more 
``quantitative'' way to find the necessary facing $4$--tuple in the cocompact case (they use the ``Flipping 
lemma'' to find a facing $4$--tuple).

 One can get the same conclusion without cocompactness, as long as there is a bound on the order of finite 
subgroups of $G$~\cite{SageevWise:tits}.  One can combine the above argument with a result of 
Caprace from~\cite{CFI} to obtain the same conclusion as in~\cite{SageevWise:tits}, but the details are 
tangential to our goal here.  Instead, we wish to highlight that, by understanding the growth of 
$\mathcal H_R$, and by 
cooking up an ``effective'' version of the Double-Skewering Lemma, one could hope to use 
Corollary~\ref{cor:linear_growth} to prove effective versions of the Tits alternative.  There is some 
discussion of this in the last section of this paper.  
\end{rem}

\section{Loxodromic isometries of $\contact X$}\label{sec:loxodromic}
The main theorem in this section is Theorem~\ref{thm:loxo}. This restates results in~\cite[Section 
5]{Hagen:boundary}; it can also be assembled from more recent results of 
Genevois~\cite{Genevois:NYJM,Genevois:acylindrical,Genevois:rank_one}, who has significantly extended the 
study of rank-one isometries of cube complexes and of contact graphs and related objects.  We give a 
simplified proof using gates and not mentioning disc diagrams.

\begin{thm}\label{thm:loxo}
 Let $X$ be a finite-dimensional, locally finite CAT(0) cube complex.  Let $g\in\Aut(X)$ be a 
combinatorially hyperbolic 
isometry of $X$.  Then the following 
are equivalent:
\begin{enumerate}
 \item \label{item:loxo} $g$ acts on $\contact X$ as a loxodromic isometry.
\item \label{item:unbounded_orbit} $\langle g\rangle$ has an unbounded orbit in $\contact X$.
 \item \label{item:rank_one_no_hyp} $g$ acts on $X$ as a rank-one isometry, and for all $n>0$ and 
all hyperplanes 
$h$ of $X$, we have $g^nh\neq h$.
\item \label{item:projection} For all $x\in X^{(0)}$ there exists $R=R(g,x)$ such that for all 
hyperplanes $h$ and all $n\in\integers$, we have $\dist(\gate_h(x),\gate_h(g^nx))\leq R$.
\end{enumerate}
Moreover, if $\langle g\rangle$ has a bounded orbit in $\contact X$, then for any hyperplane $h$ intersecting 
an axis of $g$, the orbit $\langle g\rangle\cdot h$ has diameter at most $3$ in $\contact X$.
\end{thm}

The lemmas supporting the proof of Theorem~\ref{thm:loxo} can be found in Section~\ref{subsec:supporting_lemmas} below.  Recall that to prove that $g\in G$ acts on $\contact X$ loxodromically, we 
need to prove that for some (hence any) hyperplane $h$, the quantity $\dist_{\contact X}(h,g^nh),n\in\integers$ is bounded below by a linear function of $|n|$.

\begin{proof}[Proof of Theorem~\ref{thm:loxo}]
 The implication \eqref{item:loxo}$\implies$\eqref{item:unbounded_orbit} is immediate.
 
 \textbf{The implication \eqref{item:unbounded_orbit}$\implies$\eqref{item:rank_one_no_hyp}:} Suppose that 
$\langle g\rangle$ has an unbounded orbit in $\contact X$.  So, $g^nh\neq h$ for all hyperplanes $h$ and all 
$n>0$.  It remains to check that $g$ is rank-one.  Since 
$g$ does not have a power stabilising a hyperplane, no axis of $g$ can lie in a neighbourhood of a hyperplane, 
since $X$ is locally finite.  Suppose that $g$ is not rank-one and let $A$ be a combinatorial geodesic axis 
for $g$.  Lemma~\ref{lem:rank_one} below implies that for all $N>0$, there are edges $e_N,f_N$ of $A$ such 
that $\dist(e_N,f_N)>N$ and $\dist_{\contact X}(\pi(e_N),\pi(h_N))=1$, since by the lemma the hyperplanes 
$h_N,v_N$ 
respectively dual to $e_N,f_N$ intersect.  Now, if $a,b$ are arbitrary hyperplanes intersecting $A$, we can 
choose $e_N,f_N$ so that the subpath of $A$ between $e_N,f_N$ contains the edges dual to $a$ and $b$.  Now, 
each of $a,b$ must cross either $h_N$ or $v_N$, so $\dist_{\contact X}(a,b)\leq 3$.  Hence $\diam(\pi(A))\leq 
3$, i.e. $\langle g\rangle$ has a bounded orbit in $\contact X$, a contradiction. 

For \textbf{the implication \eqref{item:rank_one_no_hyp}$\implies$\eqref{item:projection}}, assume that $g$ is 
rank-one and has no power stabilising a hyperplane.  Fix a combinatorial geodesic axis $A$ of $g$ and let $Y$ 
be its 
convex hull. By Lemma~\ref{lem:rank_one}, 
there exists $R_0<\infty$ such that $Y\subseteq\neb_{R_0}(A)$.  We will show that there exists $R_1$ such that 
$\diam(\gate_Y(h))=\diam(\gate_h(Y))\leq R_1$ for all hyperplanes $h$.  From this, we get 
assertion~\eqref{item:projection} as 
follows: let $x\in X^{(0)}$ and let $y=\gate_Y(x)$.  For any hyperplane $h$ and any $n\in\integers$, we have 
$$\dist(\gate_h(x),\gate_h(g^nx))\leq 
\dist(\gate_h(y),\gate_h(g^ny))+\dist(\gate_h(x),\gate_h(y))+\dist(\gate_h(g^nx),\gate_h(g^ny)),$$ by the 
triangle inequality.  Since $\gate_h$ is $1$--lipschitz, 
$\dist(\gate_h(x),\gate_h(y))\leq\dist(x,y)=\dist(x,Y)$ and 
$$\dist(\gate_h(g^nx),\gate_h(g^ny))\leq \dist(g^nx,g^ny)=\dist(x,y)=\dist(x,Y).$$

Hence $\dist(\gate_h(x),\gate_h(g^nx))\leq R_1+2\dist(x,Y)$, and \eqref{item:projection} holds with 
$R(g,x)=R_1+2\dist(x,Y)$.  So, it remains to produce $R_1$.

Suppose that no such $R_1$ exists, so that for all $N>0$, there exists a hyperplane $h_N$ with 
$\diam(\gate_Y(h_N))>N$.  Fix a base $0$--cube $a\in A$.  

Now, if $h$ is a hyperplane separating $h_N$ from 
$Y$, we have $\diam(\gate_Y(h_N))\leq\diam(\gate_Y(h))$, and 
we can replace $h_N$ by $h$.  Hence we can assume that no hyperplane separates $h_N$ from $Y$.  Thus 
$\neb(h_N)$ contains a point $x_N$ lying in $Y$.  Note that $\dist(x_N,A)\leq R_0$.  

So, by translating by an appropriate power of $g$ and enlarging $R_0$ by an amount depending on $g$, we can 
assume $x_N$ satisfies $\dist(a,x_N)\leq R_0$ for the base $0$--cube $a\in A$ chosen above.  Since $X$ is locally finite, only finitely 
many hyperplanes intersect the $R_0$--ball about $a$.  

So $\{h_N\}_{N>0}$ is finite, whence there exists a 
hyperplane $h$ such that $\gate_Y(h)$ is unbounded and $a$ is $R_0$--close to $h$.  Now, 
since $Y\cap\neb(h)\neq\emptyset$, we have $\gate_Y(\neb(h))=Y\cap\neb(h)$.  Thus $Y\cap\neb(h)$ is 
unbounded. Now, $Y\subset\neb_{R_0}(A)$.  Hence $A$ has a sub-ray $A'$ lying in the $R_0$--neighbourhood of 
$\neb(h)$.  In other words, for all $n\ge0$ (say), we have $\dist(h,g^na)\leq R_0$.  Hence, for all $n$ and 
all $0\leq i\leq n$, we have $\dist(g^ih,g^na)\leq R_0$.  Let $K$ 
be the number of hyperplanes crossing the $R_0$--ball about $a$.  Then for $n>K$, the list $h,gh,\cdots,g^nh$ 
must contain two identical elements, by the pigeonhole principle, so $h=g^ih$ for some $i\neq 0$, a 
contradiction.  This completes the proof that \eqref{item:rank_one_no_hyp}$\implies$\eqref{item:projection}.

Now \textbf{we prove \eqref{item:projection}$\implies$\eqref{item:loxo}}.  Recall that we need to 
verify that 
for each hyperplane $h$, there exists $C>0$ such that $\dist_{\contact X}(h,g^nh)\geq Cn$ for all 
$n>0$.

Let $h$ be a hyperplane and let $x\in\neb(h)$.  Fix $n>0$ and consider the vertices $x$ and $g^nx$.  By 
Lemma~\ref{lem:hierarchy} below, there exists a combinatorial geodesic $\gamma=\gamma_1\cdots\gamma_k$ joining 
$x$ to $g^nx$ and having the following properties:
\begin{itemize}
 \item there is a sequence $h=h_1,\ldots,h_k=g^nh$ of hyperplanes such that 
$\neb(h_i)\cap\neb(h_{i+1})\neq\emptyset$ for $1\leq i\leq k-1$;
\item $h=h_1,\ldots,h_k=g^nh$ is a geodesic of $\contact X$;
\item the geodesic $\gamma_i$ lies in $\neb(h_i)$ for $1\leq i\leq k$;
\item $\gamma_i$ has length at most $\dist(\gate_{\neb(h_i)}(x),\gate_{\neb(h_i)}(g^nx))$.
\end{itemize}
By \eqref{item:projection} and the fourth bullet point above, there exists $R=R(g,x)\geq 1$ such that 
$|\gamma_i|\leq R$ for all $i$.  Hence 
$k\geq \frac{1}{R}\dist(x,g^nx)$ for all $n$.  On the other hand, since $g$ is combinatorially 
hyperbolic, there exists $\tau\geq 1$ (depending on $g$) such that 
$\dist(x,g^nx)\geq \tau n$ for all $n$.  So, taking $C=\tau/2R$ completes the proof.

To conclude, we prove the \textbf{``moreover'' clause}.  Suppose that $\langle g\rangle$ has a bounded orbit 
in $\contact X$, so that by the equivalence of \eqref{item:unbounded_orbit} and \eqref{item:rank_one_no_hyp}, 
either $g^nh=h$ for some hyperplane $h$ and some $n>0$, or $g$ is not rank-one (or both).  We saw above that 
if the former does not hold, and $g$ is not rank-one, then $\langle g\rangle\cdot v$ has diameter 
at most 
$3$ in $\contact X$ for any hyperplane $v$ crossing a $g$--axis. 

If the former holds, then since $g^n$ stabilises $h$, the carrier $\neb(h)$ contains an axis $B$ of $g^n$. Any hyperplane $v$ crossing $B$ crosses $h$, whence 
$\dist_{\contact X}(v,g^jv)\leq 2$ for all $j\in\integers$.  Any $g$--axis is also a $g^n$--axis, and any two axes of $g^n$ cross the same hyperplanes.  So any hyperplane $v$ intersecting any 
$g$--axis must intersect $B$ and thus satisfy $\dist_{\contact X}(v,g^jv)\leq 2$ for all $j$, as required.
%
\end{proof}

\begin{rem}[The local finiteness hypothesis and other proofs in the literature]\label{rem:loc_finite}
We can obtain a similar conclusion without the local finiteness hypothesis.  Specifically, 
Lemmas~\ref{lem:hierarchy} and~\ref{lem:rank_one} do not use that assumption: the former works for 
\emph{arbitrary} CAT(0) cube complexes, and the latter uses only that $X$ is finite-dimensional.  Provided $X$ 
is finite-dimensional, but with no local finiteness hypothesis, a variant of the above proof gives that 
\eqref{item:loxo},\eqref{item:unbounded_orbit},\eqref{item:projection} are all equivalent.

There are other ways to prove parts of Theorem~\ref{thm:loxo} from results in the literature.  For example, 
equivalence of \eqref{item:projection} and \eqref{item:loxo} can easily be deduced from a result of 
Genevois~\cite[Proposition 4.2]{Genevois:NYJM}.  Meanwhile, it is straightforward to show 
\eqref{item:unbounded_orbit} implies \eqref{item:projection}, and \eqref{item:loxo} implies 
\eqref{item:unbounded_orbit} is obvious.
\end{rem}

\begin{rem}[Non-equivariant versions]
One can state a non-equivariant version of Theorem~\ref{thm:loxo} about projecting geodesic rays to $\contact 
X$; see Section 2 of~\cite{Hagen:boundary}.  Here is a simple version.  Let $X$ be a CAT(0) cube 
complex (with no local finiteness or dimension assumption).  Let $\gamma:[0,\infty)\to X$ be a combinatorial 
geodesic ray.  Then $\pi\circ\gamma:[0,\infty)\to\contact X$ is a quasigeodesic if and only if there exists 
$R$ such that $\diam(\gate_{\neb(h)}(\gamma))\leq R$ for all hyperplanes $h$.  Indeed, given such a bound, 
Lemma~\ref{lem:hierarchy} implies that $\pi\circ\gamma$ uniformly fellow-travels a geodesic ray in $\contact 
X$.  On the other hand, if $\diam(\gate_{\neb(h)}(\gamma))$ is unbounded, then $\gamma$ has arbitrarily large 
subpaths projecting to stars in $\contact X$, so $\pi\circ\gamma$ is not a (parameterised) quasigeodesic.  
Such a $\pi\circ\gamma$ still has image lying at finite Hausdorff distance from a geodesic ray or segment in 
$\contact X$, again by Lemma~\ref{lem:hierarchy}.
\end{rem}

Here are some corollaries.  The first appears in~\cite[Section 5]{Hagen:boundary}, but we reproduce it since 
we will use it in 
Section~\ref{sec:strengthened_sector_lemma}.

\begin{cor}[$\contact X$ loxodromics skewer specified pairs]\label{cor:long_skewer}
Let $X$ be a finite-dimensional CAT(0) cube complex.  Let $G\to\Aut(X)$ be an essential group action.  Assume 
that one of 
the following holds: $G$ does not fix a point in $\boundary X$, or $X$ is locally finite and $G$ acts 
cocompactly. 

Suppose that $h,v$ are hyperplanes of $X$ satisfying $\dist_{\contact 
X}(h,v)>3$.  Then there exists a hyperbolic isometry $g\in G$ such that $g$ acts loxodromically on 
$\contact X$ and $v$ separates $h$ from $gh$. 
\end{cor}

\begin{proof}
 By the Double-Skewering Lemma, there exists a hyperbolic isometry $g\in\Aut(X)$ of 
$X$ such that $v$ separates $h$ from $gh$.  It follows that $\dist_{\contact X}(h,gh)>3$, so $g$ is loxodromic 
on $\contact X$ by Theorem~\ref{thm:loxo}.
\end{proof}

Hyperplanes $h,v$ are \emph{strongly separated} if they are disjoint and no hyperplane intersects both.  This 
notion is due to Behrstock-Charney~\cite{BehrstockCharney} and plays an important role in Caprace-Sageev's 
rank rigidity theorem~\cite{CapraceSageev}.  For generalisations and applications of this property, see, for 
example,~\cite{CFI,ChatterjiMartin,Genevois:acylindrical,Levcovitz,CharneySultan}.

\begin{cor}[Strong separation criterion]\label{cor:strong_separation}
Let $X$ be a finite-dimensional CAT(0) cube complex.  Let $G\to\Aut(X)$ be an essential group action.  Assume 
that one of 
the following holds: $G$ does not fix a point in $\boundary X$, or $X$ is locally finite and $G$ acts 
cocompactly. 

 Suppose that $h,v$ are strongly separated hyperplanes of $X$.  Then there 
exists a hyperbolic isometry $g\in G$ such that $g$ acts loxodromically on 
$\contact X$ and $v$ separates $h$ from $gh$ (i.e. $g$ double-skewers the pair $h,v$). 
\end{cor}

\begin{rem}
 As mentioned above, a result of Genevois implies that if $v,h$ is a pair of strongly separated hyperplanes 
and $g$ is an isometry of $X$ double-skewering the pair $v,h$ (i.e. $g\OL h\subsetneq \OL v\subsetneq \OL h$), 
then $g$ is loxodromic on $\contact X$~\cite[Proposition 4.2]{Genevois:NYJM}.  Conversely, if $g$ is 
loxodromic on $\contact X$, and $h$ is a hyperplane crossing some (hence any) axis of $g$, then we can choose 
$n>0$ such that $\dist_{\contact X}(h,g^nh)>2$, so $h,g^nh$ are strongly separated.  Moreover, $g^{2n}h$ and 
$h$ are separated by $g^nh$, since $g$ translates along the axis (preserving the orientation), and thus we 
have a strongly-separated pair skewered by a power of $g$.
\end{rem}

\begin{proof}[Proof of Corollary~\ref{cor:strong_separation}]
Apply Double-Skewering to $h,v$ to obtain a hyperbolic element $g$ double-skewering the pair $h,v$; in 
particular, $v$ separates $h,gh$ (as required by the statement).

We will show that $\dist_{\contact X}(h,g^4h)>3$.  Since $g$ double-skewers $h$ and $g^4h$, it will then follow from Theorem~\ref{thm:loxo} (as in the proof of Corollary~\ref{cor:long_skewer}) that 
$g$ is $\contact X$--loxodromic.

Suppose, for the sake of a contradiction, that $\dist_{\contact X}(h,g^4h)\leq 3$. Consider the hyperplanes $h,gh,g^2h,g^3h,g^4h$. Let $h=u_1,u_2,u_3,u_4=g^4h$ be the 
vertex 
sequence of a $\contact X$--path of length at most $3$ from $h$ to $g^4h$. Then for some $j\leq 4$ and some 
$i\leq 3$, we have that $u_j$ crosses $g^ih$ and $g^{i+1}h$, and hence crosses $g^iv$ and $g^ih$.  But then 
$g^{-i}u_j$ crosses $h$ and $v$, contradicting strong separation.
%
\end{proof}

\begin{cor}[Rank-rigidity and the contact graph]\label{cor:rank_rigidity}
 Let $X$ be a finite-dimensional, irreducible, locally finite, essential CAT(0) cube complex, and suppose $X$ has at least one hyperplane (i.e. $X$ is not a single vertex).  Suppose that 
the action of the group $G$ on $X$ is cocompact.  Then $G$ contains an element $g$ acting loxodromically on 
$\contact X$.
\end{cor}

\begin{proof}
Since $X$ is locally finite and $G$ acts cocompactly, Corollary 4.9 of~\cite{CapraceSageev} applies ($X$ is unbounded since it contains at least one hyperplane and is essential; the action is 
\emph{hereditarily essential}, in the language of~\cite{CapraceSageev}, since $X$ is locally finite and $G$ acts cocompactly).  Since $X$ is irreducible, we therefore have one of the following:
\begin{enumerate}
 \item $G$ does not fix a point in the visual boundary $\boundary X$.  In this case, Proposition~5.1 of~\cite{CapraceSageev} says that $X$ contains a strongly separated pair of hyperplanes.   

\item All hyperplanes of $X$ are compact.  Since $G$ acts cocompactly, there are finitely many orbits of hyperplanes.  So we can choose $D\geq 0$ that exceeds the diameter of every hyperplane.  
Since $X$ is unbounded, it contains hyperplanes $h,v$ with $\dist(h,v)>D$.  Such a pair must be strongly separated.
\end{enumerate}
In either case, $X$ contains a strongly separated pair 
of hyperplanes, and we conclude by applying Corollary~\ref{cor:strong_separation}.
\end{proof}

\subsection{Supporting lemmas}\label{subsec:supporting_lemmas}
The following lemma is known (compare e.g.~\cite[Proposition 4.2]{Genevois:rank_one}), but we 
prove it here for self-containment:

\begin{lem}[Convex hulls of axes]\label{lem:rank_one}
 Let $X$ be a finite-dimensional CAT(0) cube complex and let $g\in\Aut(X)$ be combinatorially hyperbolic.  Let 
$A$ be a combinatorial 
geodesic axis for $g$, and let $Y$ be the cubical convex hull of $A$.  Then:
\begin{enumerate}
 \item \label{item:no_regular}If no regular neighbourhood of $A^{(0)}$ contains $Y^{(0)}$, the element $g$ is 
not rank-one.
 \item \label{item:rank_one}If some regular neighbourhood of $A^{(0)}$ contains $Y^{(0)}$, then either $g$ is 
rank-one or $A$ lies 
in a regular neighbourhood of a hyperplane.
\end{enumerate}
In particular, if $g$ is not rank-one and $A$ does not lie in a neighbourhood of a hyperplane, then for all 
$N\geq 0$, there exist hyperplanes $h,v$, dual to edges $e_h,e_v$ of $A$, such that $\dist(e_h,e_v)>N$ but 
$h\cap v\neq\emptyset$.
\end{lem}

Before the proof of Lemma~\ref{lem:rank_one}, we need two auxiliary lemmas:

\begin{lem}\label{lem:hull_of_axis}
Let $X,g,A,Y$ be as in Lemma~\ref{lem:rank_one}.  For $r\geq 0$, let $A_r$ be the subgeodesic from $A(-r)$ to $A(r)$ and let $Y_r$ be the cubical convex hull of $A_r$.  Then $Y=\bigcup_{r\ge0}Y_r$.  
\end{lem}

\begin{proof}
 Since any convex subcomplex containing $A$ must contain $A_r$ for all $r$, we have $\bigcup_{r\ge0}Y_r\subseteq Y$.  To prove the reverse inclusion, let $y\in Y$ be a vertex.  Choose $r\ge 0$ such 
that every hyperplane separating $y$ from $A(0)$ separates $A(-r)$ from $A(r)$; this is possible since each hyperplane separating $y$ from $A(0)$ 
crosses $Y$ and hence $A$, and since there are finitely many such hyperplanes.  Any halfspace containing 
$A_r$ contains $A(-r),A(0),A(r)$, and hence the associated hyperplane does not separate $A(0)$ from $y$.  Thus 
$y\in Y_r$.  Hence $Y=\bigcup_{r\ge0}Y_r$.
\end{proof}

In the next lemma, by a \emph{half-flat} in the CAT(0) cube complex $X$, we mean an isometric embedding $F:[0,\infty)\times\reals\to X$ where $X$ 
is given the CAT(0) metric and $[0,\infty)\times\reals$ is given the Euclidean metric.  We also use the notation $F$, and the term ``half-flat'', for its image in $X$.  The bi-infinite CAT(0) 
geodesic $\beta$ \emph{bounds} the 
half-flat $F$ if $F|_{\{0\}\times\reals}=\beta$ (here we are also using the notation $\beta$ both for the isometric embedding $\beta:\reals\to X$ and for its image).

\begin{lem}[Hyperplanes and half-flats]\label{lem:half_flat_description}
 Let $X$ be a finite-dimensional CAT(0) cube complex.  Let $\beta$ be a CAT(0) geodesic that bounds a half-flat $F$.  Let $h$ be a hyperplane of $X$ such that $h\cap F\neq\emptyset$.  
Then one of the following holds:
\begin{itemize}
 \item $F\subseteq h$;
 \item $F\cap h$ is a line in $F$ parallel to $\beta$;
 \item $F\cap h$ is a ray in $F$ whose initial point is on $\beta$.
\end{itemize}
%
%
%
%
%
\end{lem}

\begin{proof}
This is an exercise in CAT(0) geometry using CAT(0) 
convexity of hyperplanes and halfspaces and the product structure of hyperplane carriers; see e.g.~\cite[Remark 3.4]{HagenPrzytycki}.
\end{proof}

\begin{proof}[Proof of Lemma~\ref{lem:rank_one}]
We first prove assertion~\eqref{item:no_regular}.  Fix $y\in Y^{(0)}$.  We claim that $y$ lies on some combinatorial geodesic with endpoints on $A$.  Indeed, by Lemma~\ref{lem:hull_of_axis}, we can 
choose $r\ge0$ such that $y\in Y_r$, where $Y_r$ is the convex hull of the subgeodesic of $A$ from $A(-r)$ to $A(r)$.  Since $Y_r$ is, equivalently, the interval between $A(-r)$ and $A(r)$, we have 
that $y$ lies on a geodesic from $A(-r)$ to $A(r)$.

%

%

Suppose that $Y$ does not lie in a regular neighbourhood of $A$.  Then for any $R>0$, the above argument shows 
that there is a combinatorial geodesic that has endpoints on $A$ but does not lie in $\neb_R(A)$.  Hence $A$ 
is not a \emph{Morse geodesic} (see e.g.~\cite[Definition 1.2]{ACGH} for the definition; the notion goes 
back in some form to~\cite{Morse}).  

Now, $Y$ is a proper CAT(0) space, because it is the convex hull of a bi-infinite combinatorial 
geodesic.\footnote{This follows from~\cite[Theorem 1.14]{BCGNW}, and one can also prove it using 
Corollary~\ref{cor:linear_growth} and the fact that $Y$ contains no facing triple.}  Moreover, $\langle 
g\rangle$ acts on $Y$ by isometries, with $g$ acting hyperbolically.  Finally, any CAT(0) geodesic axis in $Y$ 
for $g$ is not Morse.  Hence, by~\cite[Lemma 3.3]{Sultan} and~\cite[Theorem 5.4]{BestvinaFujiwara}, $g$ is not 
rank-one.

We now prove assertion~\eqref{item:rank_one} and the ``in particular'' statement.

For both, we suppose that $g$ is not rank-one and $A$ does not lie in any neighbourhood of any hyperplane. Let $\beta:\reals\to Y$ be a CAT(0) geodesic axis of $g$, and let $F$ be a half-flat bounded 
by $\beta$.  (The half-flat $F$ need not be unique or $g$--invariant, but we will not need this.)  

Let $h$ be a hyperplane intersecting $F$.  By 
Lemma~\ref{lem:half_flat_description}, $h\cap F$ is either all of $F$, a line parallel to $\beta$, or a ray starting on $\beta$.  In either of the first two cases, $\beta$ is contained in a 
neighbourhood of $h$, so, since $A$ and $\beta$ lie at finite Hausdorff distance, the same is true for $A$, a contradiction.  Hence $h$ intersects $\beta$ in a point and intersects $F$ in a ray.

From this, we obtain assertion~\eqref{item:rank_one} as follows.  We have seen that every hyperplane crossing $F$ crosses $\beta$, and hence crosses $A$, since $A$ and $\beta$ cross the same 
hyperplanes.  Let $Z$ be the cubical convex hull of $F$.  Any hyperplane crossing $Z$ crosses $F$, and hence crosses $Y$.

Consider the gate map $\gate_Y:X\to Y$.  We claim that there exists $N_0\in\naturals$ such that for all $z\in Z^{(0)}$, we have $\dist_2(z,\gate_Y(z))\leq N_0$.  Indeed, let $w$ be a hyperplane 
separating $z$ from $\gate_Y(z)$.  Recall from Section~\ref{sec:defn} that such a $w$ must separate $z$ from $Y$.  Since every hyperplane crossing $Z$ crosses $Y$, we see that $w$ must not cross $Z$, 
and thus separates all of $Z$ from $Y$.  
There are finitely many such hyperplanes, so Lemma~\ref{lem:QI} provides the constant $N_0$. 

Let $N\geq 0$ be given.  Choose $f\in F$ such that $\dist_2(f,\beta)>N$, which is possible since $F$ is a half-flat bounded by $\beta$.  Choose a $0$--cube $f'\in Z$ such that $\dist_2(f,f')\leq 
\sqrt{\dimension X}/2.$  So $\gate_Y(f')\in Y$ satisfies $$\dist_2(f,\gate_Y(f'))\leq\sqrt{\dimension X}/2+N_0.$$  Hence $$\dist_2(\gate_Y(f'),\beta)>N-\sqrt{\dimension X}/2-N_0.$$  Since $\beta$ 
and $A$ are at finite Hausdorff distance, and 
we could have chosen $N$ arbitrarily large, $Y$ contains $0$--cubes arbitrarily far from $A$, as required.  (We have used the CAT(0) metric here but the same conclusion applies in the 
$\ell_1$ metric by Lemma~\ref{lem:QI}.)

To prove the ``in particular'' statement, let $h$ be a hyperplane intersecting $F$, and recall that $h\cap F$ is a ray starting on $\beta$.  Let $\rho=h\cap F$ be such a ray, and let 
$t\in\reals$ be such that $\rho\cap\beta=\beta(t)$.  Since $\rho$ is a CAT(0) geodesic ray in $X$, for all $n\in\naturals$ there exists a hyperplane $h_n$ such that $\dist(h_n\cap \rho,\beta(t))>n$.  
Now, since $h_n$ intersects $F$, it does so in a ray $\nu_n$ intersecting $\beta$ at a point $\beta(t_n)$.  Since any interval in $\beta$ intersects finitely many hyperplanes (by, say, 
Lemma~\ref{lem:QI}), the quantity $|t_n|$ is unbounded as $n\to\infty$.  Hence, for any $r\geq 0$, there exist hyperplanes $h,v$ such that $h\cap v\neq\emptyset$ but 
$\dist_2(h\cap\beta,v\cap\beta)>r$.  Since $\beta$ and $A$ fellow travel (in the CAT(0) and $\ell_1$ metrics), and cross the same hyperplanes, we obtain the desired conclusion.  
\end{proof}

The next lemma is Lemma 3.1 in~\cite{BHS:HHS_I}.  Here we give essentially the same proof, except using gates 
instead of disc diagrams.  

\begin{lem}[``Hierarchy paths'' in $\contact X$]\label{lem:hierarchy}
Let $X$ be a CAT(0) cube complex and let $x,y\in X^{(0)}$.  Let $h_x,h_y$ be hyperplanes such that 
$x\in\neb(h_x)$ and $y\in\neb(h_y)$.  Then there exists a sequence $h_x=h_1,\ldots,h_k=h_y$ of hyperplanes 
such that:
\begin{itemize}
 \item $\neb(h_i)\cap\neb(h_{i+1})\neq\emptyset$ for $1\leq i\leq k-1$, and 
 \item $h_1,\ldots,h_k$ is a $\contact X$--geodesic 
from $h_x$ to $h_y$;
\item there is a path $\gamma=\gamma_1\cdots\gamma_k$ from $x$ to $y$, where each $\gamma_i$ is a 
combinatorial geodesic in $\neb(h_i)$;
\item $\gamma$ is a combinatorial geodesic;
\item for each $i$, $|\gamma_i|\leq \dist(\gate_{\neb(h_i)}(x),\gate_{\neb(h_i)}(y))$.
\end{itemize}
Hence $x$ and $y$ are joined by a geodesic $\gamma$ such that $\pi\circ\gamma$ is an unparameterised 
quasigeodesic in $\contact X$ (with constants independent of $X$).  
\end{lem}

\begin{rem}
Lemma~\ref{lem:hierarchy} says roughly that any two points in $X$ can be joined by a geodesic that tracks a 
geodesic between their projections to the contact graph.  This is reminiscent of ``hierarchy paths'' in the 
marking complex of a surface~\cite{MMII}, with the curve graph playing the role of the contact graph.  This 
similarity is part of the motivation for the notion of a \emph{hierarchically hyperbolic 
space}~\cite{BHS:HHS_I}.
\end{rem}

\begin{proof}[Proof of Lemma~\ref{lem:hierarchy}]
 Let $h_x=h_1,\ldots,h_k=h_y$ be a sequence of hyperplanes satisfying the first two properties, which is 
possible just because $\contact X$ is a connected graph.  

Let 
$x=x_1\in\neb(h_1)$.  For $2\leq i\leq k$, suppose that $x_{i-1}\in\neb(h_{i-1})$ has been chosen, and let 
$x_i=\gate_{\neb(h_i)}(x_{i-1})\in\neb(h_i)$.  Let $x_{k+1}=y$.  For each $i$, let $\gamma_i$ be a 
combinatorial geodesic in $\neb(h_i)$ joining $x_i$ to $x_{i+1}$.  Let $\gamma=\gamma_1\cdots\gamma_k$.

The \emph{complexity} of the pair $((h_1,\ldots,h_k),(\gamma_1,\ldots,\gamma_k))$ of $k$--tuples is the tuple 
$(|\gamma_1|,\ldots,|\gamma_k|)$, taken in lexicographic order.  Suppose that $\gamma_1,\ldots,\gamma_k$ have 
been chosen as above so as to minimise the complexity.  

\textbf{$\gamma$ is a geodesic:}  We claim 
that $\gamma$ is a combinatorial geodesic.  Suppose to the contrary that some hyperplane $h$ is dual to two 
distinct edges of $\gamma$, respectively lying in $\gamma_i,\gamma_j$ for $1\leq i\leq j\leq k$.  We cannot 
have $i=j$, since $\gamma_i$ is a geodesic.  We also cannot have $j>i+2$, because $h$ intersects $\neb(h_i)$ 
and $\neb(h_j)$, which would yield a path $h_1,\ldots,h_i,h,h_j,\ldots,h_k$ in $\contact X$ from $h_1$ to 
$h_k$.  This path has length less than $k-1$, contradicting that $h_1,\ldots,h_k$ is a geodesic.

Hence $j=i+1$ or $j=i+2$.  If $j=i+1$, then $h$ intersects $\neb(h_i)$ and $\neb(h_{i+1})$, and thus does not 
separate 
$x_i$ from $\neb(h_{i+1})$.  Hence $h$ does not separate $x_i$ from $\gate_{\neb(h_{i+1})}(x_i)$, which is a 
contradiction since $h$ is dual to an edge of $\gamma_i$, and therefore separates the endpoints of $\gamma_i$. 
 Thus 
$j\neq i+1$.

If $j=i+2$, then we can replace $h_{i+1}$ by $h$ to yield a lower-complexity pair.  Indeed, we replace 
$h_1,\ldots,h_k$ by $h_1,\ldots,h_i,h,h_{i+2},\ldots,h_k$, obtaining a new geodesic in $\contact X$ from $h_x$ 
to $h_y$.  
We replace $x_{i+1}$ by $x_{i+1}'=\gate_{\neb(h)}(x_i)$ and replace $x_{i+2}$ by 
$\gate_{\neb(h_{i+2})}(x_{i+1}')$.  The geodesics $\gamma_1,\ldots,\gamma_{i-1}$ are unchanged, but $\gamma_i$ 
is replaced by a geodesic of length 
$$\dist(x_i,\gate_{\neb(h)}(x_i))=\dist(x_i,\neb(h))<\dist(x_i,\neb(h_{i+1}))=|\gamma_i|,$$ so we have reduced 
complexity.  This contradicts our initial choice of pair of $k$--tuples, and we conclude that $j\neq i+2$. 
 Hence no hyperplane is dual to two distinct edges of $\gamma$, so $\gamma$ is a geodesic.

\textbf{Length of $\gamma_i$:}  Fix $i$ and let $h$ be a hyperplane crossing $\gamma_i$.  Since $\gamma$ is a 
geodesic, $h$ separates $x$ from $y$.  On the other hand, $h$ crosses the convex subcomplex $\neb(h_i)$, so 
$h$ must separate $\gate_{\neb(h_i)}(x)$ from $\gate_{\neb(h_i)}(y)$, by Lemma~\ref{lem:gate_sep}.  Since 
$|\gamma_i|$ is the number of 
hyperplanes $h$ crossing $\gamma_i$, we have $\dist(\gate_{\neb(h_i)}(x),\gate_{\neb(h_i)}(y))\geq |\gamma_i|$.

\textbf{Unparameterised quasigeodesic:}  Let $\gamma$ be a geodesic from $x$ to $y$ provided by the first part 
of the lemma, so that $\gamma=\gamma_1\cdots\gamma_k$, where each $\gamma_i$ lives in the carrier of a 
hyperplane $h_i$, and the sequence $h_1,\ldots,h_k$ is a $\contact X$--geodesic from a point $h_1\in\pi(x)$ to 
a point $h_k\in\pi(y)$.  By construction, $\pi(\gamma_i)$ lies in the $1$--neighbourhood in $\contact X$ of 
$h_i$, so $\pi\circ\gamma$ lies at uniformly bounded Hausdorff distance in $\contact X$ from some, and hence 
any, geodesic from $\pi(x)$ to $\pi(y)$, as required.
\end{proof}

\section{The sector lemma}\label{sec:strengthened_sector_lemma}

Our goal is to prove Proposition~\ref{propi:SSL}.  Fix a CAT(0) cube complex $X$ satisfying the hypotheses: $X$ is irreducible, locally finite, essential, hyperplane-essential, and $\Aut(X)$ acts 
cocompactly (hence $X$ is finite-dimensional).  We continue to use the convention that, if $h$ is a hyperplane, then $\OL h,\OR h$ denote the associated halfspaces.

The main lemma is:

\begin{lem}\label{lem:sector_hop}
Let $h,v$ be distinct hyperplanes of $X$ such that $h\cap v\neq\emptyset$.  Suppose that $\OL h\cap \OL v$ contains a hyperplane $a$ such that $\dist_{\contact X}(a,h)>2$.  Then $\OL h\cap \OR v$ 
contains a hyperplane $a'$ such that $\dist_{\contact X}(a',h)>2$.
\end{lem}

\begin{proof}
Since $\dist_{\contact X}(a,h)>2$, Lemma~\ref{lem:single_point} implies that $\gate_h(a)$ is a single point, which we denote by $p$.  

Now, $h$ is a locally finite CAT(0) cube complex, and 
$\stabilizer_{\Aut(X)}(h)$ acts on $h$ cocompactly by Lemma~\ref{lem:hereditary_cocompactness}, since $\Aut(X)$ acts on $X$ cocompactly.  Since the action of $\Aut(X)$ on $X$ is assumed to be 
hyperplane-essential, the action of $\stabilizer_{\Aut(X)}(h)$ on $h$ is essential, so by Proposition~3.2 of~\cite{CapraceSageev}, there exists 
$g\in\stabilizer_{\Aut(X)}(h)$ such that, regarded as a hyperplane of $h$, the intersection $h\cap v$ 
separates $g^{-1}(h\cap v)$ from $g(h\cap v)$, and $g\OR v\subsetneq \OR v$.  By replacing $g$ by $g^2$ if necessary, we can assume that $g$ stabilises $\OL h$ and $\OR h$.  By replacing $g$ with a 
further positive power, we have $gp\in \OR v$.  In other words, $g\gate_h(a)=\gate_{gh}(ga)=\gate_h(ga)$ is contained in $\OR v$.  Hence, since $v$ crosses $h$, we have $ga\subset \OR v$.  

Since $g$ stabilises $\OL h$, we get 
$ga\subset \OL h\cap \OR v$, as required.  Moreover, $\dist_{\contact X}(ga,h)=\dist_{\contact X}(ga,gh)=\dist_{\contact X}(a,h)>2$, so taking $a'=ga$ completes the proof.  
\end{proof}

Now we can prove the proposition.

\begin{proof}[Proof of Proposition~\ref{propi:SSL}]
The proof is summarised in Figure~\ref{fig:prop1}.

Since $X$ is irreducible, Corollary~\ref{cor:rank_rigidity} implies that $\contact X$ is unbounded.  Hence 
there exists a hyperplane $a$ such that $\dist_{\contact X}(a,h)>2$.  Since $\dist_{\contact X}(h,v)=1$, we have $\dist_{\contact X}(a,v)>1$.  So, up to relabelling halfspaces, we have $a\subset \OL 
h\cap \OL v$.  Applying Lemma~\ref{lem:sector_hop}, we obtain a hyperplane $a'$ such that $\dist_{\contact X}(h,a')>2$ and $a'\subset\OL h\cap\OR v$.  

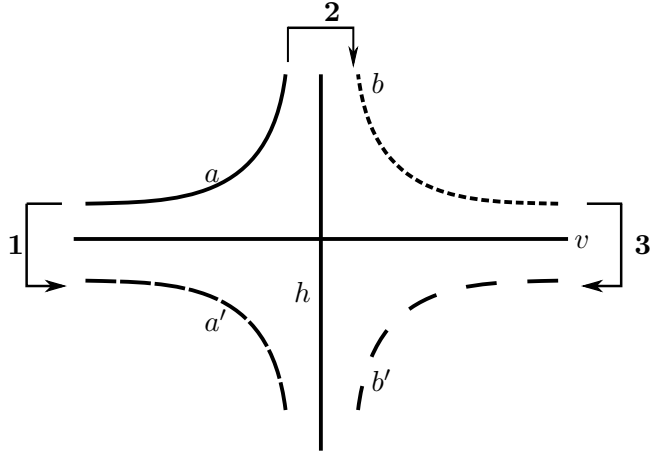
\begin{figure}[h]
 \begin{overpic}[width=0.5\textwidth]{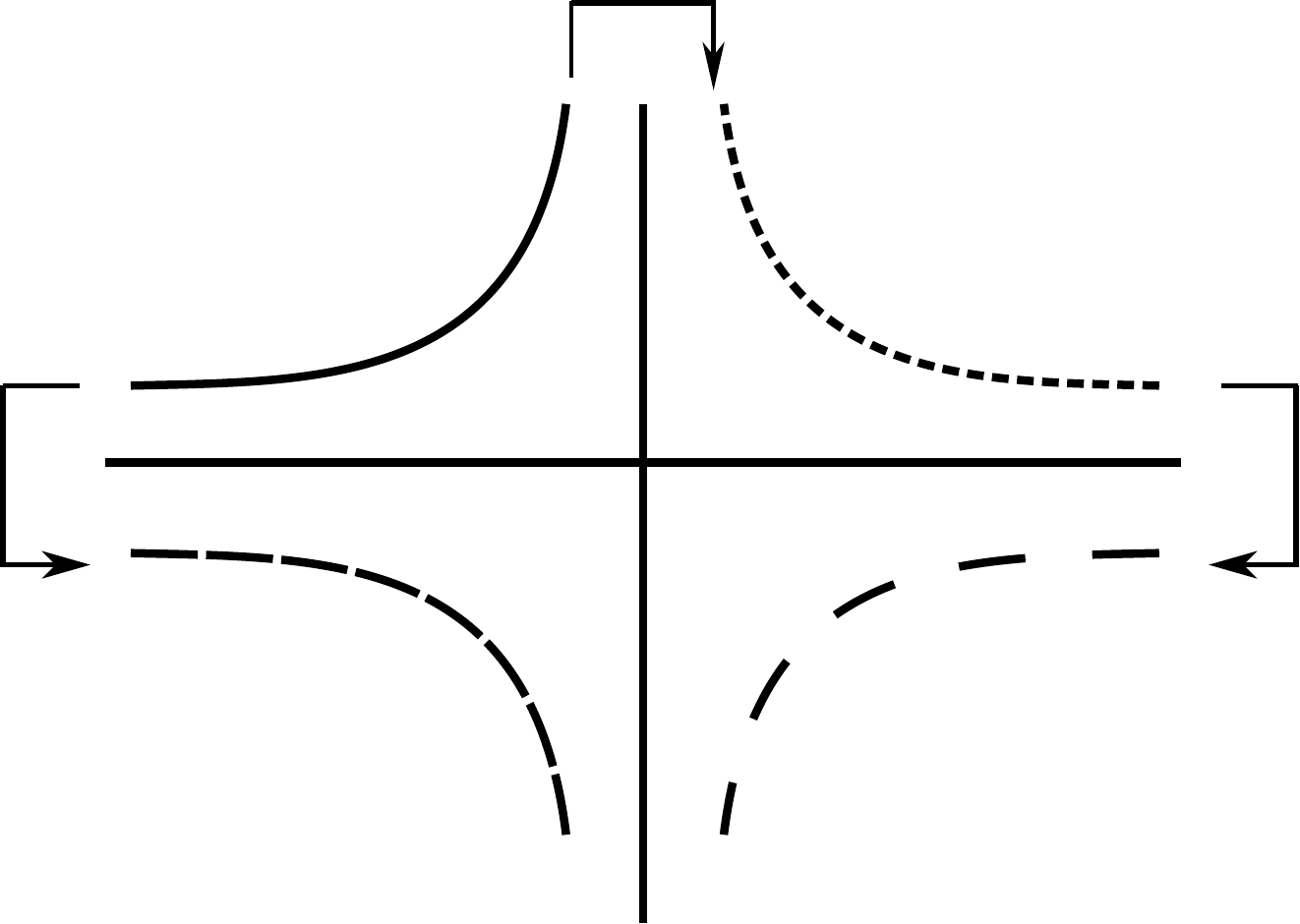}
  \put(30,45){$a$}
  \put(30,20){$a'$}
  \put(58,60){$b$}
  \put(58,10){$b'$}
  \put(45,25){$h$}
  \put(92,34){$v$}
  \put(-3,33){$\mathbf 1$}
 \put(50,72){$\mathbf 2$}
 \put(102,33){$\mathbf 3$}
 \end{overpic}
\caption{Proof of Proposition~\ref{propi:SSL}. Start with the hyperplane $a$.  The arrow labelled $\mathbf 1$ shows the first application of   
Lemma~\ref{lem:sector_hop} to $a$. The arrow labelled $\mathbf 2$ shows the application of Corollary~\ref{cor:strong_separation} to the strongly separated pair $a,h$. The arrow labelled $\mathbf 3$ 
shows the application of Lemma~\ref{lem:sector_hop} to the hyperplane $b$.  The hyperplane $b$ provided by Corollary~\ref{cor:strong_separation} could also have been in the bottom right sector, in 
which case we would use Lemma~\ref{lem:sector_hop} to move $b$ ``up'' into the top right.}\label{fig:prop1}
\end{figure}

Apply Corollary~\ref{cor:strong_separation} to the strongly separated pair $a, h$ to obtain a hyperplane $b\subset \OR h$ with $\dist_{\contact X}(b,h)>2$ (the hyperplane $b$ arises as a translate 
of $a$ by some element acting loxodromically on $\contact X$ and double-skewering $a$ and $h$).  Again, $\dist_{\contact X}(b,v)>1$, so 
$b\subset \OL v$ or $b\subset\OR v$.  Suppose the former holds, i.e. $b\subset \OR h\cap \OL v$.  Applying Lemma~\ref{lem:sector_hop} (with the roles of $\OL h$ and $\OR h$ switched), we 
obtain a hyperplane $b'\subset \OR h\cap \OR v$.  By an identical argument, if $b\subset \OR h\cap\OR v$, we obtain a hyperplane $b'\subset\OR h\cap \OL v$.  We have shown that each of the four 
intersections $\OL h\cap\OL v,\OL h\cap\OR v,\OR h\cap\OL v,\OR h\cap\OR 
v$ contains a hyperplane.  One of these four intersections is $h^+\cap v^+$, so we are done.  
\end{proof}

\section{Questions}\label{sec:questions}
We close with some questions about using Proposition~\ref{prop:ramsey_dilworth} to 
effectivise statements about actions on CAT(0) cube complexes.

\begin{question}[Effective double skewering]\label{question:EDS}
 Let the group $G$ act cocompactly and essentially on the finite-dimensional, locally finite CAT(0) cube 
complex $X$.  Find an explicit estimate of the function $f:\naturals\to\naturals$ such that the following holds: 
let $v,h$ be disjoint hyperplanes and let $L=\dist(\neb(v),\neb(h))$.  Then there exists $g\in G$ such that 
$v$ separates $h$ from $gh$ and $g$ has combinatorial translation length at most $f(L)$.

The function $f$ should be allowed to depend on invariants of $X$ like its dimension and the maximum degree of 
$0$--cubes.  It also seems reasonable (and necessary) to allow $f$ to depend on the number of orbits of 
hyperplanes, or the diameter of a smallest compact convex subcomplex whose $G$--orbit covers $X$, or some 
other parameter of the action. 
\end{question}

Next, recall that, in the proof of Corollary~\ref{cor:linear_growth}, when $X$ is finite-dimensional, we found a constant $K$, depending on the dimension of $X$, 
such that for all $R\ge0$, either the set $\mathcal H_R$ of hyperplanes crossing the $R$--ball about a fixed 
basepoint contains a facing $4$--tuple, or $|\mathcal H_R|\leq KR$.

So, if one knew the growth rate of the function $R\mapsto|\mathcal H_R|$, and this growth rate was 
superlinear, one could compute a minimal $R_0$ such that $\mathcal H_{R_0}$ contains a facing $4$--tuple. In 
conjunction with an answer to Question~\ref{question:EDS}, the proof of Proposition~\ref{prop:tits_1} would 
then yield $g,h\in G$, each with translation length at most $f(2R_0)$, such that $g,h$ freely generate a free 
subgroup of $G$.  So:

\begin{question}[Effective Tits alternative 1]\label{question:ETA1}
 Let $X$ be a locally finite, finite-dimensional CAT(0) cube complex on which the group $G$ acts cocompactly and essentially.  
Find a constant $L_0$ such that either $R\mapsto|\mathcal H_R|$ grows 
at most linearly, or $G$ contains a free group generated by elements $g,h$ whose combinatorial translation 
lengths are at most $L_0$.  The constant $L_0$ should depend on specific parameters of the $G$--action in an 
explicit way, as in Question~\ref{question:EDS}.
\end{question}

\begin{rem}
The proof of the Tits alternative in~\cite{CapraceSageev} also involves an application of the Double Skewering 
Lemma to two pairs of hyperplanes drawn from a facing $4$--tuple.  In that setting, the facing $4$--tuple is 
found by applying the \emph{Flipping Lemma} to a facing triple, and then concluding that, if $X$ has no facing 
triple, then the $G$--action fixes a point at infinity.  So, it would also be interesting to try to 
effectivise the Flipping Lemma (of which the Double Skewering Lemma is an easy consequence).
\end{rem}

Our proof of Proposition~\ref{propi:SSL} also yields a facing $4$--tuple of hyperplanes $a,b,c,d$ that are 
pairwise strongly separated (because they are pairwise at large distance in $\contact X$, by construction).  
Applying the Double Skewering Lemma as in the proof of Proposition~\ref{prop:tits_1} would then yield a free 
subgroup of $G$ generated by two elements acting on $\contact X$ loxodromically, in view of 
Corollary~\ref{cor:strong_separation}.  

\begin{question}[Effective rank rigidity]\label{question:ERR}
Is there an effective version of the Rank-Rigidity Theorem for cocompact actions on CAT(0) cube complexes?  
Specifically: let $X$ be a finite-dimensional, essential, cocompact, locally finite CAT(0) cube complex, let 
$x_0\in X^{(0)}$ be a base vertex.  For $R\ge 0$, let $B_R(x_0)$ and $\mathcal H_R$ be defined as above.  Let 
$vol(R)$ be the number of $0$--cubes in $B_R(x_0)$ and let $Hvol(R)=|\mathcal H_R|$.  Can one characterise, in 
terms of the growth rates of $vol(R)$ and $Hvol(R)$, when $X$ splits as a nontrivial product?  If $X$ does not 
split as a nontrivial product, can one estimate the value $R_0$, depending on $vol(R)$ and 
$Hvol(R)$, such that $\mathcal H_{R_0}$ contains a strongly separated pair of hyperplanes?

Given a cocompact action of $G$ on $X$, one could then combine this with an answer to 
Question~\ref{question:EDS} and Corollary~\ref{cor:strong_separation} to produce a rank-one element of bounded 
translation length.
\end{question}

One can imagine a more complicated version of Question~\ref{question:ERR} about facing $4$--tuples of strongly 
separated hyperplanes, and a free subgroup of $G$ acting purely loxodromically on $\contact X$, generated by 
two elements of bounded translation length on $X$.  One can also imagine a version about the translation 
lengths on $\contact X$, rather than on $X$.  In fact:

\begin{question}\label{question:CXtrans}
 Let $X$ be a finite-dimensional, locally finite, irreducible, essential CAT(0) cube complex with a group $G$ 
acting cocompactly.  Estimate the minimal translation length on $\contact X$ of elements of $G$ acting 
loxodromically.  
\end{question}

We finally ask whether this is related to \emph{uniform exponential growth} for cubulated groups.  Here the 
question boils down to: let the finitely generated group $G$ act properly and cocompactly on the CAT(0) cube 
complex $X$.  Find a constant $\lambda$ such that either $G$ is virtually abelian or, for any finite 
generating set of $G$, the $\lambda$--ball in the corresponding Cayley graph of $G$ contains two elements that 
freely generate a free (semi)group.  There is quite a strong result about this in the $2$--dimensional case, 
due to Kar and Sageev~\cite{KarSageev:UEG}, and this is an actively-studied question in higher dimensions.

\bibliographystyle{alpha}
\bibliography{sector_lemma}

\newcommand{\etalchar}[1]{$^{#1}$}
\begin{thebibliography}{AKWW13}

\bibitem[ACGH17]{ACGH}
Goulnara~N. Arzhantseva, Christopher~H. Cashen, Dominik Gruber, and David Hume.
\newblock Characterizations of {M}orse quasi-geodesics via superlinear
  divergence and sublinear contraction.
\newblock {\em Doc. Math.}, 22:1193--1224, 2017.

\bibitem[AGM13]{agol2013virtual}
Ian Agol, Daniel Groves, and Jason Manning.
\newblock The virtual {H}aken conjecture.
\newblock {\em Doc. Math}, 18:1045--1087, 2013.

\bibitem[AKWW13]{AlgomKfir}
Yael Algom-Kfir, Bronislaw Wajnryb, and Pawel Witowicz.
\newblock A parabolic action on a proper, {CAT}(0) cube complex.
\newblock {\em J. Group Theory}, 16(6):965--984, 2013.

\bibitem[AOS12]{AOS}
Federico Ardila, Megan Owen, and Seth Sullivant.
\newblock Geodesics in {$\rm CAT(0)$} cubical complexes.
\newblock {\em Adv. in Appl. Math.}, 48(1):142--163, 2012.

\bibitem[Ava61]{avann1961median}
Sherwin~P. Avann.
\newblock Median algebras.
\newblock {\em Proceedings of the American Mathematical Society}, 12:407--414,
  1961.

\bibitem[BC93]{barthelemy1993median}
Jean-Pierre Barth{\'e}lemy and Julien Constantin.
\newblock Median graphs, parallelism and posets.
\newblock {\em Discrete mathematics}, 111(1-3):49--63, 1993.

\bibitem[BC08]{bandelt2008metric}
Hans-Jurgen Bandelt and Victor Chepoi.
\newblock Metric graph theory and geometry: a survey.
\newblock {\em Contemporary Mathematics}, 453:49--86, 2008.

\bibitem[BC12]{BehrstockCharney}
Jason Behrstock and Ruth Charney.
\newblock Divergence and quasimorphisms of right-angled {A}rtin groups.
\newblock {\em Math. Ann.}, 352(2):339--356, 2012.

\bibitem[BCG{\etalchar{+}}09]{BCGNW}
Jacek Brodzki, Sarah~J. Campbell, Erik Guentner, Graham~A. Niblo, and Nick~J.
  Wright.
\newblock Property {A} and {$\rm CAT(0)$} cube complexes.
\newblock {\em J. Funct. Anal.}, 256(5):1408--1431, 2009.

\bibitem[BF09]{BestvinaFujiwara}
Mladen Bestvina and Koji Fujiwara.
\newblock A characterization of higher rank symmetric spaces via bounded
  cohomology.
\newblock {\em Geom. Funct. Anal.}, 19(1):11--40, 2009.

\bibitem[BF18]{BeyrerFioravanti:hyperbolic}
Jonas Beyrer and Elia Fioravanti.
\newblock Cross ratios and cubulations of hyperbolic groups.
\newblock {\em arXiv preprint arXiv:1810.08087}, 2018.

\bibitem[BF19]{BeyrerFioravanti:marked}
Jonas Beyrer and Elia Fioravanti.
\newblock Cross ratios on {CAT}(0) cube complexes and marked length-spectrum
  rigidity.
\newblock {\em arXiv preprint arXiv:1903.02447}, 2019.

\bibitem[BH99]{BH}
Martin~R. Bridson and Andr\'{e} Haefliger.
\newblock {\em Metric spaces of non-positive curvature}, volume 319 of {\em
  Grundlehren der Mathematischen Wissenschaften [Fundamental Principles of
  Mathematical Sciences]}.
\newblock Springer-Verlag, Berlin, 1999.

\bibitem[BHS17a]{BHS:HHS_I}
Jason Behrstock, Mark~F. Hagen, and Alessandro Sisto.
\newblock Hierarchically hyperbolic spaces, {I}: {C}urve complexes for cubical
  groups.
\newblock {\em Geom. Topol.}, 21(3):1731--1804, 2017.

\bibitem[BHS17b]{behrstock2017quasiflats}
Jason Behrstock, Mark~F Hagen, and Alessandro Sisto.
\newblock Quasiflats in hierarchically hyperbolic spaces.
\newblock {\em arXiv preprint arXiv:1704.04271}, 2017.

\bibitem[Bow13]{bowditch2013coarse}
Brian~H. Bowditch.
\newblock Coarse median spaces and groups.
\newblock {\em Pacific Journal of Mathematics}, 261(1):53--93, 2013.

\bibitem[Bow18]{bowditch2018convex}
Brian~H. Bowditch.
\newblock Convex hulls in coarse median spaces.
\newblock {\em Preprint}, 2018.

\bibitem[Bow19]{bowditch2019quasiflats}
Brian~H. Bowditch.
\newblock Quasiflats in coarse median spaces, 2019.

\bibitem[Bre17]{Bregman}
Corey Bregman.
\newblock Isometry groups of {CAT}(0) cube complexes.
\newblock {\em arXiv preprint arXiv:1712.04805}, 2017.

\bibitem[Bri91]{Bridson:thesis}
Martin~R. Bridson.
\newblock Geodesics and curvature in metric simplicial complexes.
\newblock In {\em Group theory from a geometrical viewpoint ({T}rieste, 1990)},
  pages 373--463. World Sci. Publ., River Edge, NJ, 1991.

\bibitem[BW12]{BergeronWise}
Nicolas Bergeron and Daniel~T. Wise.
\newblock A boundary criterion for cubulation.
\newblock {\em Amer. J. Math.}, 134(3):843--859, 2012.

\bibitem[CC19]{ChepoiChalopin}
J{\'e}r{\'e}mie Chalopin and Victor Chepoi.
\newblock 1-safe petri nets and special cube complexes: equivalence and
  applications.
\newblock {\em ACM Transactions on Computational Logic (TOCL)}, 20(3):1--49,
  2019.

\bibitem[CFI16]{CFI}
Indira Chatterji, Talia Fern\'{o}s, and Alessandra Iozzi.
\newblock The median class and superrigidity of actions on {$\rm CAT(0)$} cube
  complexes.
\newblock {\em J. Topol.}, 9(2):349--400, 2016.
\newblock With an appendix by Pierre-Emmanuel Caprace.

\bibitem[CH17]{CordesHume}
Matthew Cordes and David Hume.
\newblock Stability and the {M}orse boundary.
\newblock {\em J. Lond. Math. Soc. (2)}, 95(3):963--988, 2017.

\bibitem[Che00]{Chepoi}
Victor Chepoi.
\newblock Graphs of some {${\rm CAT}(0)$} complexes.
\newblock {\em Adv. in Appl. Math.}, 24(2):125--179, 2000.

\bibitem[CM19]{ChatterjiMartin}
Indira Chatterji and Alexandre Martin.
\newblock A note on the acylindrical hyperbolicity of groups acting on {${\rm
  CAT}(0)$} cube complexes.
\newblock In {\em Beyond hyperbolicity}, volume 454 of {\em London Math. Soc.
  Lecture Note Ser.}, pages 160--178. Cambridge Univ. Press, Cambridge, 2019.

\bibitem[CN05]{ChatterjiNiblo}
Indira Chatterji and Graham Niblo.
\newblock From wall spaces to {$\rm CAT(0)$} cube complexes.
\newblock {\em Internat. J. Algebra Comput.}, 15(5-6):875--885, 2005.

\bibitem[CS11]{CapraceSageev}
Pierre-Emmanuel Caprace and Michah Sageev.
\newblock Rank rigidity for {CAT}(0) cube complexes.
\newblock {\em Geom. Funct. Anal.}, 21(4):851--891, 2011.

\bibitem[CS15]{CharneySultan}
Ruth Charney and Harold Sultan.
\newblock Contracting boundaries of {CAT}(0) spaces.
\newblock {\em Journal of Topology}, 8(1):93--117, 2015.

\bibitem[Dil50]{Dilworth}
Robert~P. Dilworth.
\newblock A decomposition theorem for partially ordered sets.
\newblock {\em Ann. of Math. (2)}, 51:161--166, 1950.

\bibitem[FH19]{FioravantiHagen}
Elia Fioravanti and Mark Hagen.
\newblock Deforming cubulations of hyperbolic groups.
\newblock {\em arXiv preprint arXiv:1912.10999}, 2019.

\bibitem[Fio17]{Fioravanti:dilworth}
Elia Fioravanti.
\newblock Roller boundaries for median spaces and algebras.
\newblock {\em arXiv preprint arXiv:1708.01005}, 2017.

\bibitem[Fio20]{Fioravanti:contact}
Elia Fioravanti.
\newblock Automorphisms of contact graphs of {CAT}(0) cube complexes.
\newblock {\em arXiv preprint arXiv:2001.08493}, 2020.

\bibitem[Gen16]{Genevois:acylindrical}
Anthony Genevois.
\newblock Acylindrical action on the hyperplanes of a {CAT}(0) cube complex.
\newblock {\em arXiv preprint arXiv:1610.08759}, 2016.

\bibitem[Gen19a]{Genevois:NYJM}
Anthony Genevois.
\newblock Acylindrical hyperbolicity from actions on {CAT}(0) cube complexes: a
  few criteria.
\newblock {\em New York J. Math}, 25:1214--1239, 2019.

\bibitem[Gen19b]{Genevois:axis}
Anthony Genevois.
\newblock A cubical flat torus theorem and some of its applications.
\newblock {\em arXiv preprint arXiv:1902.04883}, 2019.

\bibitem[Gen19c]{Genevois:rank_one}
Anthony Genevois.
\newblock Rank-one isometries of {CAT}(0) cube complexes and their
  centralisers.
\newblock {\em arXiv preprint arXiv:1905.00735}, 2019.

\bibitem[Gen20]{Genevois:diagram}
Anthony Genevois.
\newblock Contracting isometries of {CAT}(0) cube complexes and acylindrical
  hyperbolicity of diagram groups.
\newblock {\em Algebraic \& Geometric Topology}, 20(1):49--134, 2020.

\bibitem[Gro81]{Gromov:polynomial}
Mikhael Gromov.
\newblock Groups of polynomial growth and expanding maps.
\newblock {\em Inst. Hautes \'{E}tudes Sci. Publ. Math.}, (53):53--73, 1981.

\bibitem[Gro87]{Gromov:essay}
Mikhail Gromov.
\newblock Hyperbolic groups.
\newblock In {\em Essays in group theory}, volume~8 of {\em Math. Sci. Res.
  Inst. Publ.}, pages 75--263. Springer, New York, 1987.

\bibitem[Hag07]{HaglundSemisimple}
Fr{\'e}d{\'e}ric Haglund.
\newblock Isometries of {CAT}(0) cube complexes are semi-simple.
\newblock {\em arXiv preprint arXiv:0705.3386}, 2007.

\bibitem[Hag08]{Haglund:graph_product}
Fr\'{e}d\'{e}ric Haglund.
\newblock Finite index subgroups of graph products.
\newblock {\em Geom. Dedicata}, 135:167--209, 2008.

\bibitem[Hag13]{Hagen:boundary}
Mark~F. Hagen.
\newblock The simplicial boundary of a {CAT}(0) cube complex.
\newblock {\em Algebr. Geom. Topol.}, 13(3):1299--1367, 2013.

\bibitem[Hag14]{Hagen:contact}
Mark~F. Hagen.
\newblock Weak hyperbolicity of cube complexes and quasi-arboreal groups.
\newblock {\em J. Topol.}, 7(2):385--418, 2014.

\bibitem[HK18]{huang2018groups}
Jingyin Huang and Bruce Kleiner.
\newblock Groups quasi-isometric to right-angled {A}rtin groups.
\newblock {\em Duke Mathematical Journal}, 167(3):537--602, 2018.

\bibitem[HP15]{HagenPrzytycki}
Mark~F. Hagen and Piotr Przytycki.
\newblock Cocompactly cubulated graph manifolds.
\newblock {\em Israel J. Math.}, 207(1):377--394, 2015.

\bibitem[HT19]{HagenTouikan}
Mark~F. Hagen and Nicholas W.~M. Touikan.
\newblock Panel collapse and its applications.
\newblock {\em Groups Geom. Dyn.}, 13(4):1285--1334, 2019.

\bibitem[Hua17]{huang2017top}
Jingyin Huang.
\newblock Top-dimensional quasiflats in cat (0) cube complexes.
\newblock {\em Geometry \& Topology}, 21(4):2281--2352, 2017.

\bibitem[HW08]{haglund2008special}
Fr{\'e}d{\'e}ric Haglund and Daniel~T Wise.
\newblock Special cube complexes.
\newblock {\em Geometric and Functional Analysis}, 17(5):1551--1620, 2008.

\bibitem[HW14]{HruskaWise}
G.~Christopher Hruska and Daniel~T. Wise.
\newblock Finiteness properties of cubulated groups.
\newblock {\em Compos. Math.}, 150(3):453--506, 2014.

\bibitem[HW20]{HagenWilton}
Mark Hagen and Henry Wilton.
\newblock Guirardel cores for cubical actions.
\newblock {\em In preparation}, 2020.

\bibitem[K\"50]{Konig}
D\'{e}nes K\"{o}nig.
\newblock {\em Theorie der endlichen und unendlichen {G}raphen.
  {K}ombinatorische {T}opologie der {S}treckenkomplexe}.
\newblock Chelsea Publishing Co., New York, N. Y., 1950.

\bibitem[KM12]{KahnMarkovic}
Jeremy Kahn and Vladimir Markovic.
\newblock Immersing almost geodesic surfaces in a closed hyperbolic three
  manifold.
\newblock {\em Ann. of Math. (2)}, 175(3):1127--1190, 2012.

\bibitem[KS19]{KarSageev:UEG}
Aditi Kar and Michah Sageev.
\newblock Uniform exponential growth for {CAT}(0) square complexes.
\newblock {\em Algebr. Geom. Topol.}, 19(3):1229--1245, 2019.

\bibitem[Lea13]{Leary}
Ian~J. Leary.
\newblock A metric {K}an-{T}hurston theorem.
\newblock {\em J. Topol.}, 6(1):251--284, 2013.

\bibitem[Lev18]{Levcovitz}
Ivan Levcovitz.
\newblock Divergence of {$\rm CAT(0)$} cube complexes and {C}oxeter groups.
\newblock {\em Algebr. Geom. Topol.}, 18(3):1633--1673, 2018.

\bibitem[Mie14]{Miesch:injective}
Benjamin Miesch.
\newblock Injective metrics on cube complexes.
\newblock {\em arXiv preprint arXiv:1411.7234}, 2014.

\bibitem[MM00]{MMII}
Howard~A. Masur and Yair~N. Minsky.
\newblock Geometry of the complex of curves. {II}. {H}ierarchical structure.
\newblock {\em Geom. Funct. Anal.}, 10(4):902--974, 2000.

\bibitem[Mor24]{Morse}
Harold~Marston Morse.
\newblock A fundamental class of geodesics on any closed surface of genus
  greater than one.
\newblock {\em Trans. Amer. Math. Soc.}, 26(1):25--60, 1924.

\bibitem[Nic04]{Nica}
Bogdan Nica.
\newblock Cubulating spaces with walls.
\newblock {\em Algebr. Geom. Topol.}, 4:297--309, 2004.

\bibitem[NPW81]{nielsen1981petri}
Mogens Nielsen, Gordon Plotkin, and Glynn Winskel.
\newblock Petri nets, event structures and domains, part i.
\newblock {\em Theoretical Computer Science}, 13(1):85--108, 1981.

\bibitem[NS13]{NevoSageev}
Amos Nevo and Michah Sageev.
\newblock The {P}oisson boundary of {${\rm CAT}(0)$} cube complex groups.
\newblock {\em Groups Geom. Dyn.}, 7(3):653--695, 2013.

\bibitem[QRT19]{Qing}
Yulan Qing, Kasra Rafi, and Giulio Tiozzo.
\newblock Sublinearly morse boundary {I}: Cat (0) spaces.
\newblock {\em arXiv preprint arXiv:1909.02096}, 2019.

\bibitem[Ram29]{Ramsey}
Frank~P. Ramsey.
\newblock On a {P}roblem of {F}ormal {L}ogic.
\newblock {\em Proc. London Math. Soc. (2)}, 30(4):264--286, 1929.

\bibitem[Rol98]{Roller}
Martin Roller.
\newblock Poc sets, median algebras and group actions.
\newblock {\em arXiv preprint arXiv:1607.07747}, 1998.

\bibitem[Sag95]{Sageev95}
Michah Sageev.
\newblock Ends of group pairs and non-positively curved cube complexes.
\newblock {\em Proc. London Math. Soc. (3)}, 71(3):585--617, 1995.

\bibitem[Sag97]{sageev1997codimension}
Michah Sageev.
\newblock Codimension-1 subgroups and splittings of groups.
\newblock {\em Journal of Algebra}, 189(2):377--389, 1997.

\bibitem[Sag14]{Sageev:book}
Michah Sageev.
\newblock {$\rm CAT(0)$} cube complexes and groups.
\newblock In {\em Geometric group theory}, volume~21 of {\em IAS/Park City
  Math. Ser.}, pages 7--54. Amer. Math. Soc., Providence, RI, 2014.

\bibitem[Sul14]{Sultan}
Harold Sultan.
\newblock Hyperbolic quasi-geodesics in {CAT}(0) spaces.
\newblock {\em Geom. Dedicata}, 169:209--224, 2014.

\bibitem[SW05]{SageevWise:tits}
Michah Sageev and Daniel~T. Wise.
\newblock The {T}its alternative for {${\rm CAT}(0)$} cubical complexes.
\newblock {\em Bull. London Math. Soc.}, 37(5):706--710, 2005.

\bibitem[Wis12]{Wise:book}
Daniel~T. Wise.
\newblock {\em From riches to {raag}s: 3-manifolds, right-angled {A}rtin
  groups, and cubical geometry}, volume 117 of {\em CBMS Regional Conference
  Series in Mathematics}.
\newblock Published for the Conference Board of the Mathematical Sciences,
  Washington, DC; by the American Mathematical Society, Providence, RI, 2012.

\bibitem[Wis20]{Wise:pave}
Daniel~T Wise.
\newblock The structure of groups with a quasiconvex hierarchy.
\newblock {\em Preprint}, 2020.

\bibitem[Woo17]{Woodhouse:axis}
Daniel~J. Woodhouse.
\newblock A generalized axis theorem for cube complexes.
\newblock {\em Algebr. Geom. Topol.}, 17(5):2737--2751, 2017.

\bibitem[Wri12]{wright2012finite}
Nick Wright.
\newblock Finite asymptotic dimension for {CAT}(0) cube complexes.
\newblock {\em Geometry \& Topology}, 16(1):527--554, 2012.

\bibitem[WW17]{WoodhouseWise}
Daniel~T. Wise and Daniel~J. Woodhouse.
\newblock A cubical flat torus theorem and the bounded packing property.
\newblock {\em Israel J. Math.}, 217(1):263--281, 2017.

\end{thebibliography}
\end{document}